\documentclass[11pt]{amsart}
\usepackage{amssymb}
\usepackage{amsthm}
\usepackage{verbatim}
\usepackage{hyperref}
\usepackage{color}
\usepackage{enumitem}

\newtheorem{prop}{Proposition}
\newtheorem{proposition}{Proposition}
\newtheorem{lemma}[prop]{Lemma}
\newtheorem{cor}[prop]{Corollary}
\newtheorem{conj}[prop]{Conjecture}
\newtheorem{theorem}[prop]{Theorem}

\newcommand{\R}{\mathbb R}
\newcommand{\Q}{\mathbb Q}
\newcommand{\Z}{\mathbb Z}
\newcommand{\C}{\mathbb C}
\newcommand{\N}{\mathbb N}
\newcommand{\T}{\mathbb T}

\newcommand{\w}{\omega}
\newcommand{\e}{\varepsilon}
\newcommand{\g}{\gamma}
\newcommand{\p}{\varphi}
\newcommand{\s}{\psi}

\renewcommand{\a}{\alpha}
\renewcommand{\b}{\beta}
\renewcommand{\d}{\delta}

\renewcommand{\P}{\mathbb{P}}

\newcommand{\supp}{\mathrm{supp}}

\newcommand{\mc}{\mathcal}

\def\set4{\mathcal I}
\def\tup14{(1,2,3,4)}

\newcommand\vwidehat[1]{\arraycolsep=0pt\relax%
\begin{array}{c}
\stretchto{
  \scaleto{
    \scalerel*[\widthof{\ensuremath{#1}}]{\kern-.5pt\bigwedge\kern-.5pt}
    {\rule[-\textheight/2]{1ex}{\textheight}} 
  }{\textheight} %
}{0.5ex}\\           
#1\\                 
\rule{-1ex}{0ex}
\end{array}
}
\newtheorem*{comm*}{Comment}
\newtheorem{definition}{Definition}

\newtheorem{corollary}{Corollary}
\newtheorem{thm*}{Theorem}[section]

\usepackage{bbm}   
\usepackage{scalerel,stackengine}
\stackMath
\newcommand\widecheck[1]{%
\savestack{\tmpbox}{\stretchto{%
  \scaleto{%
    \scalerel*[\widthof{\ensuremath{#1}}]{\kern-.6pt\bigwedge\kern-.6pt}%
    {\rule[-\textheight/2]{1ex}{\textheight}}
  }{\textheight}%
}{0.5ex}}%
\stackon[1pt]{#1}{\scalebox{-1}{\tmpbox}}%
}

\def\XXint#1#2#3{{\setbox0=\hbox{$#1{#2#3}{\int}$ }
\vcenter{\hbox{$#2#3$ }}\kern-.6\wd0}}

\usepackage[utf8]{inputenc}
\usepackage{calc}
\usepackage{accents}

\begin{document}

\title{Sharp superlevel set estimates for small cap decouplings of the parabola}
\author{Yuqiu Fu, Larry Guth, and Dominique Maldague}

\date{\today}

\maketitle

\begin{abstract} We prove sharp bounds for the size of superlevel sets $\{x\in \R^2:|f(x)|>\a\}$ where $\a>0$ and $f:\R^2\to\C$ is a Schwartz function with Fourier transform supported in an $R^{-1}$-neighborhood of the truncated parabola $\P^1$. These estimates imply the small cap decoupling theorem for $\P^1$ from \cite{smallcap} and the canonical decoupling theorem for $\P^1$ from \cite{BD}. New $(\ell^q,L^p)$ small cap decoupling inequalities also follow from our sharp level set estimates. 

\end{abstract}




In this paper, we further develop the high/low frequency proof of decoupling for the parabola \cite{gmw} to prove sharp level set estimates which recover and refine the small cap decoupling results for the parabola in \cite{smallcap}. We begin by describing the problem and our results in terms of exponential sums. The main results in full generality are in \textsection\ref{techsec}. 




For $N\ge 1$, $R\in[N,N^2]$, and $2\le p$, let $D(N,R,p)$ denote the smallest constant so that
\begin{equation}\label{const} |Q_{R}|^{-1}\int_{Q_{R}}|\sum_{\xi\in\Xi}a_\xi e((x,t)\cdot(\xi,\xi^2))|^pdxdt\le D(N,R,p) N^{p/2}  \end{equation}
for any collection $\Xi\subset[-1,1]$ with $|\Xi|\sim N$ consisting of $\sim\frac{1}{N}$-separated points, $a_\xi\in\C$
with $|a_\xi|\sim 1$, and any cube $Q_{R}\subset\R^2$ of sidelength $R$.

A corollary of the small cap decoupling theorem for the parabola in \cite{smallcap} is that if $2\le p\le 2+2s$ for $R=N^s$, then 
\begin{equation}\label{small} D(N,R,p)\le C_\e N^\e.   \end{equation}
This estimate is sharp, up to the $C_\e N^\e$ factor, which may be seen by Khintchine's inequality. The range $2\le p\le 2+2s$ is the largest range of $p$ for which $D(N,R,p)$ may be bounded by sub-polynomial factors in $N$. 
The case $R=N^2$ of \eqref{small} follows from the canonical $\ell^2$ decoupling theorem of Bourgain and Demeter for the parabola \cite{BD}. For $R<N^2$ and the subset $\Xi=\{k/N\}_{k=1}^N$, the inequality \eqref{const} is an estimate for the moments of exponential sums over subsets smaller than the full domain of periodicity (i.e. $N^2$ in the $t$-variable). Bourgain investigated examples of this type of inequality in \cite{smallcapb1,smallcapb2}.

By a pigeonholing argument (see Section 5 of \cite{gmw}), \eqref{small} follows from upper bounds for superlevel sets $U_\a$ defined by
\[ U_\a=\{(x,t)\in\R^2: |\sum_{\xi\in\Xi}a_\xi e((x,t)\cdot(\xi,\xi^2))|>\a\} .  \]
In particular, \eqref{small} is equivalent, up to a $\log N$ factor, to proving that for any $\a>0$ and for $R=N^s$,
\begin{equation}\label{smalllvlset} \a^{2+2s} |U_\a\cap Q_R|\le C_\e R^\e N^{1+s}R^2\end{equation}
when $\Xi$, $a_\xi$ satisfy the hypotheses following \eqref{const}. In this paper, we improve the above superlevel set estimate for all $\a>0$ strictly between $N^{1/2}$ and $N$. 

\begin{theorem} \label{mainconjexp} Let $R\in[N,N^2]$. For any $\e>0$, there exists $C_\e<\infty $ such that 
\[|U_\a\cap Q_{R}|\le C_\e N^\e \begin{cases} \frac{N^{2}R}{\a^4}\sum\limits_{\xi\in\Xi}|a_\xi|^2\quad&\text{if}\quad \a^2>R\\
\frac{N^{2}R^2}{\a^6}\sum\limits_{\xi\in\Xi}|a_\xi|^2\quad&\text{if}\quad N\le \a^2\le R\\
R^2 \quad&\text{if}\quad \a^2<N. 
\end{cases} \]
whenever $\Xi\subset[-1,1]$ is a $\gtrsim \frac{1}{N}$-separated subset, $|a_\xi|\le 1$ for each $\xi\in\Xi$, and $Q_{R}\subset\R^2$ is a cube of sidelength $R$. 
\end{theorem}
Our superlevel set estimates are essentially sharp, which follows from analyzing the function $F(x,t)=\sum_{n=1}^Ne((x,t)\cdot(\frac{n}{N},\frac{n^2}{N^2}))$. It is not known whether the implicit constant in the upper bound of \eqref{small} goes to infinity with $N$ except in the case that $p=6$ and $s=2$, when the same example $F(x,t)=\sum_{n=1}^Ne((x,t)\cdot(\frac{n}{N},\frac{n^2}{N^2}))$ shows that $D(N,N^2,6)\gtrsim (\log N)$ \cite{bcountex}. Roughly, the argument is that for each dyadic value $\a\in[N^{3/4},N]$, 
one can show by counting the ``major arcs" that
\[ \a^6\{(x,t)\in Q_{N^2}:|F(x,t)|\sim \a\}|\gtrsim N^4\cdot N^3 . \]
Since there are $\sim \log N$ values of $\a$, the lower bound for $\int_{Q_{N^2}}|F|^6$ follows. Theorem \ref{mainconjexp} implies that the corresponding superlevel set estimates \eqref{smalllvlset} are not sharp for $1\le s<2$, unless $\a\sim N$ or $\a^2\sim N$, which leads to the following conjecture.

\begin{conj} Let $s\in[1,2)$ and $2\le p\le 2+2s$. There exists $C(s)>0$ so that
\[ D(N,N^s,p)\le C(s). \]
\end{conj}



A more refined version of Theorem \ref{mainconjexp} leads to the following essentially sharp $(\ell^q,L^p)$ small cap decoupling theorem, stated here for general exponential sums. 
\begin{corollary}
  Let $\frac{3}{p}+\frac{1}{q}\leq 1,$ and let $R\in [N,N^2].$ Then for each $\e>0$, there exists $C_\e<\infty$ so that
  \[\|\sum_{\xi\in\Xi} a_\xi e((x,t)\cdot(\xi,\xi^2))\|_{L^p(B_R)} \leq C_\e N^\e (N^{1-\frac{1}{p}-\frac{1}{q}}R^{\frac{1}{p}}+N^{\frac{1}{2}-\frac{1}{q}}R^{\frac{2}{p}})(\sum_{\xi}|a_\xi|^q)^{1/q}. \]

\end{corollary}
In the above corollary, the assumptions are that $\Xi$ is a $\gtrsim\frac{1}{N}$-separated subset of $[-1,1]$ and that $a_\xi\in\C$.

\section{Main results  \label{techsec}}

We state our main results in the more general set-up for decoupling. Let $\P^1$ denote the truncated parabola 
\[ \{(t,t^2):|t|\le 1\} \]
and write $\mc{N}_{R^{-1}}(\P^1)$ for the $R^{-1}$-neighborhood of $\P^1$ in $\R^2$, where $R\ge 2$. For a partition $\{\g\}$ of $\mc{N}_{R^{-1}}(\P^1)$ into almost rectangular blocks, an $(\ell^2,L^p)$ decoupling inequality is
\begin{equation}\label{dec} \|f\|_{L^p(B_R)}\le D(R,p) (\sum_\g\|f_\g\|_{L^p(\R^2)}^2)^{1/2}\end{equation}
in which $f:\R^2\to\C$ is a Schwartz function with $\supp \widehat{f}\subset\mc{N}_{R^{-1}}(\P^1)$ and $f_\g$ means the Fourier projection onto $\g$, defined precisely below. When we refer to canonical caps or to canonical decoupling, we mean that $\g$ are approximately $R^{-1/2}\times R^{-1}$ blocks corresponding to the $\ell^2$-decoupling paper of \cite{BD}. 
In this paper, we allow $\g$ to be approximate $R^{-\b}\times R^{-1}$ blocks, where $\frac{1}{2}\le\b\le 1$. This is the ``small cap" regime studied in \cite{smallcap}. We also consider $(\ell^q,L^p)$ decoupling for small caps, which replaces $(\sum_\g\|f_\g\|_p^2)^{1/2}$ by $(\sum_\g\|f_\g\|_p^q)^{1/q}$ in the decoupling inequality above (see Corollary \ref{lqLp}).

To precisely discuss the collection $\{\g\}$, fix a $\b\in[\frac{1}{2},1]$. Let $\mc{P}=\mc{P}(R,\b)=\{\g\}$ be the partition of $\mc{N}_{R^{-1}}(\P^1)$ given by 
\begin{equation}\label{blocks} \bigsqcup_{|k|\le \lceil R^\b\rceil-2} \{(x,t)\in\mc{N}_{R^{-1}}(\P^1):k\lceil R^\b\rceil^{-1}\le x<(k+1)\lceil R^\b\rceil^{-1} \}  \end{equation}
and the two end pieces
\[  \{(x,t)\in\mc{N}_{R^{-1}}(\P^1):x<-1+\lceil R^\b\rceil^{-1}\} \sqcup  \{(x,t)\in\mc{N}_{R^{-1}}(\P^1):1-\lceil R^\b\rceil^{-1}\le x\}. \]
For a Schwartz function $f:\R^2\to\C$ with $\supp\widehat{f}\subset\mc{N}_{R^{-1}}(\P^1)$, define for each $\g\in\mc{P}(R,\b)$
\[ f_\g(x):=\int_\g\widehat{f}(\xi)e^{2\pi ix\cdot\xi}d\xi. \]

For $a,b>0$, the notation $a\lesssim b$ means that $a\le Cb$ where $C>0$ is a universal constant whose definition varies from line to line, but which only depends on fixed parameters of the problem. Also, $a\sim b$ means $C^{-1}b\le a\le Cb$ for a universal constant $C$. 

Let $U_\a:= \{x\in \R^2 :|f(x)|\ge \a\}$. In Section 5 of \cite{gmw}, through a wave packet decomposition and series of pigeonholing steps, bounds for $D(R,p)$ in \eqref{dec} follow (with an additional power of $(\log R)$) from bounds on the constant $C(R,p)$ in
\[ \a^p|U_\a|\le C(R,p)(\#\{\g:f_\g\not=0\})^{\frac{p}{2}-1}\sum_\g\|f_\g\|_2^2 \]
for any $\a>0$ and under the additional assumptions that $\|f_\g\|_\infty\lesssim 1$, $\|f_\g\|_p^p\sim\|f_\g\|_2^2$ for each $\g$. Thus decoupling bounds follow from upper bounds on the superlevel set $|U_\a|$. In this paper, we consider the question: given $\a>0$ and a partition $\{\g\}$, how large can $|U_\a|$ be, varying over functions $f$ satisfying $\|f_\g\|_\infty\lesssim 1$ for each $\g$? We answer this question in the following theorem.

\begin{theorem} \label{mainconj}Let $\b\in[\frac{1}{2},1]$, $R\ge2$. Let $f:\R^2\to\C$ be a Schwartz function with Fourier transform supported in $\mc{N}_{R^{-1}}(\P^1)$ satisfying $\|f_\g\|_\infty\le 1$ for all $\g\in\mc{P}(R,\b)$. Then for any $\a>0$,
\[ |U_\a\cap[-R,R]^2|\le C_\e R^\e\begin{cases} \frac{R^{2\b-1}}{\a^4}\sum\limits_\g\|f_\g\|_{L^2(\R^2)}^2\quad&\text{if}\quad \a^2>R\\
\frac{R^{2\b}}{\a^6}\sum\limits_\g\|f_\g\|_{L^2(\R^2)}^2\quad&\text{if}\quad R^\b\le \a^2\le R\\
R^2 \quad&\text{if}\quad \a^2<R^\b. 
\end{cases}\]

\end{theorem}
Each bound in Theorem \ref{mainconj} is sharp, up to the $C_\e R^\e$ factor, which we show in \textsection\ref{sharp}.

Define notation for a distribution function for the Fourier support of a Schwartz function $f$ with Fourier transform supported in $\mc{N}_{R^{-1}}(\P^1)$ as follows. For each $0\le s\le 2$, let 
\[ \lambda(s)=\sup_{\w(s)}\#\{\g:\g\cap \w(s)\not=\emptyset,\,\,f_\g\not=0\}\]
where $\w(s)$ is any arc of $\P^1$ with projection onto the $\xi_1$-axis equal to an interval of length $s$. The following theorem implies Theorem \ref{mainconj} and replaces factors of $R^\b$ in the upper bounds from Theorem \ref{mainconj} by expressions involving $\lambda(\cdot)$, which see the actual Fourier support of the input function $f$. 

\begin{theorem}\label{main} Let $\b\in[\frac{1}{2},1]$, $R\ge2$. For any $f$ with Fourier transform supported in $\mc{N}_{R^{-1}}(\P^1)$ satisfying $\|f_\g\|_\infty\lesssim 1$ for each $\g\in\mc{P}(R,\b)$,
\[ |U_\a|\le C_\e R^\e \begin{cases} \frac{1}{\a^4}\underset{s}{\max}\lambda(s^{-1}R^{-1})\lambda(s)\sum_\g\|f_\g\|_2^2\quad&\text{if}\quad \a^2>\frac{\lambda(1)^2}{\max_s \lambda(s^{-1}R^{-1})\lambda(s)}\\
\frac{\lambda(1)^2}{\a^6}\sum_\g\|f_\g\|_2^2\quad&\text{if}\quad  \a^2\le \frac{\lambda(1)^2}{\max_{s}\lambda(s^{-1}R^{-1})\lambda(s)} 
\end{cases}\]
in which the maxima are taken over dyadic $s$, $R^{-\b}\le s\le R^{-1/2}$. 
\end{theorem}

\begin{cor}[$(l^q,L^p)$ small cap decoupling] \label{lqLp}Let 
$\frac{3}{p}+\frac{1}{q}\le 1$. Then 
\[ \|f\|_{L^p(B_R)}\le C_\e R^\e (R^{\b (1-\frac{1}{q})-\frac{1}{p}(1+\b)}+R^{\b (\frac{1}{2}-\frac{1}{q})})(\sum_\g\|f_\g\|_{L^p(\R^2)}^q)^{1/q}\]
whenever $f$ is a Schwartz function with Fourier transform supported in $\mc{N}_{R^{-1}}(\P^1)$. 
\end{cor}


The powers of $R$ in the upper bound come from considering two natural sharp examples for the ratio $\|f\|_{L^p(B_R)}^p/(\sum_\g\|f_\g\|_p^q)^{p/q}$. The first is the square root cancellation example, where $|f_\g|\sim \chi_{B_R}$ for all $\g$ and $f=\sum_\g e_\g f_\g$ where $e_\g$ are $\pm1$ signs chosen (using Khintchine's inequality) so that $\|f\|_{L^p(B_R)}^p\sim R^{\b p/2}R^2$.
\[ \|f\|_p^p/(\sum_\g\|f_\g\|_p^q)^{p/q}\gtrsim (R^{\b p/2}R^2)/(R^{\b p/q}R^2)\sim R^{\b p(\frac{1}{2}-\frac{1}{q})}. \]
The second example is the constructive interference example. Let $f_\g=R^{1+\b}\widecheck{\eta}_\g$ where $\eta_\g$ is a smooth bump function approximating $\chi_\g$. Since $|f|=|\sum_\g f_\g|$ is approximately constant on unit balls and $|f(0)|\sim R^\b$, we have
\[\|f\|_p^p/(\sum_\g\|f_\g\|_p^q)^{p/q}\gtrsim (R^{\b p})/(R^{\b p/q}R^{1+\b})\sim R^{\b p(1-\frac{1}{q})-1-\b}. \]
There is one more example which may dominate the ratio: The block example is $f=R^{1+\b}\sum_{\g\subset \theta}\widecheck{\eta}_\g$ where $\theta$ is a canonical $R^{-1/2}\times R^{-1}$ block. Since $f=f_\theta$ and $|f_\theta|$ is approximately constant on dual $\sim R^{1/2}\times R$ blocks $\theta^*$, we have
\[ 
    \frac{\a^p|U_\a|}{(\#\g)^{\frac{p}{q}}\|f_\g\|_2^2}\gtrsim \frac{R^{(\b-\frac{1}{2})p}R^{\frac{3}{2}}}{R^{(\b-\frac{1}{2})\frac{p}{q}}R^{1+\b}}=R^{(\b-\frac{1}{2})p(1-\frac{1}{q})+\frac{1}{2}-\b}. \]
One may check that the constructive interference examples dominate the block example when $\frac{3}{p}+\frac{1}{q}\le 1$. 
We do not investigate $(l^q,L^p)$ small cap decoupling in the range $\frac{3}{p}+\frac{1}{q}>1$ in the present paper. 

The paper is organized as follows. In \textsection\ref{sharp}, we demonstrate that Theorem \ref{mainconj} is sharp using an exponential sum example. In \textsection\ref{implications}, we show how Theorem \ref{mainconj} follows easily from Theorem \ref{main} and how after some pigeonholing steps, so does Corollary \ref{lqLp}. Then in \textsection\ref{tools}, we develop the multi-scale high/low frequency tools we use in the proof of Theorem \ref{main}. These tools are very similar to those developed in \cite{gmw}. It appears that a more careful version of the proof of Theorem \ref{main} could also replace the $C_\e R^\e$ factor by a power of $(\log R)$, as is done for canonical decoupling in \cite{gmw}. Finally, in \textsection\ref{pf}, we prove a bilinear version of Theorem \ref{main} and then reduce to the bilinear case to finish the proof.

LG is supported by a Simons Investigator grant. DM is supported by the National Science Foundation under Award No. 2103249.

\section{A sharp example \label{sharp}}
Because we will show that Theorem \ref{main} implies Theorem \ref{mainconj}, it suffices to show that Theorem \ref{mainconj} is sharp, which we mean up to a $C_\e R^\e$ factor. 
Write $N=\lceil R^\b\rceil$. The function achieving the sharp bounds is
\[ F(x_1,x_2)=\sum_{k=1}^Ne(\frac{k}{N}x_1+\frac{k^2}{N^2}x_2)\eta(x_1,x_2),\]
where $\eta$ is a Schwartz function 
satisfying $\eta\sim 1$ on $[-R,R]^2$ and $\text{supp }\widehat{\eta}\subset B_{R^{-1}}$. We will bound the set
\[ U_\a=\{(x_1,x_2)\in[-R,R]^2:|F(x_1,x_2)|\ge \a\}. \]
\newline\noindent \underline{Case 1: $R<\a^2$.}

Suppose that $\a\sim N$ and note that $F(0,0)=N$ and $|F(0,0)|\sim N$ when $|(x_1,x_2)|<\frac{1}{10^3}$.  Using periodicity in the $x_1$ variable, there are $\sim R/N$ many other heavy balls where $|F(x)|\sim N$ in $[-R,R]^2$. For $\a$ in the range suppose that $R<\a^2< N^2$, we will show that $U_\a$ is dominated by larger  neighborhoods of the heavy balls.

Let $r=N^2/\a^2$ and assume without loss of generality that $r$ is in the range $R^\e<r<N^2/R\sim R^{2\b-1}\ll N$. The upper bound for $|U_\a|$ in Theorem \ref{mainconj} for this range is 
\[ |U_\a|\le C_\e R^\e \frac{N^2}{\a^4R}\sum_\g\|F_\g\|_2^2\sim C_\e R^\e\frac{N^2}{\a^4R}NR^2.\]
To demonstrate that this inequality is sharp, by periodicity in $x_1$, it suffices to show that $|U_\a\cap B_r|\gtrsim r^2$. 
Let 
$\phi_{r^{-1}}$ be a nonnegative bump function supported in $B_{r^{-1}/2}$ with $\phi_{r^{-1}}\gtrsim 1$ on $B_{r^{-1}/4}$. 
Let $\eta_r=r^4({\phi_{r^{-1}}*\phi_{r^{-1}}})^{\widecheck{\,\,\,\,}}$ and analyze the $L^2$ norm $\|F\|_{L^2(\eta_r)}$. By Plancherel's, 
\begin{align*}\|F\|_{L^2(\eta_r)}^2=\int|F|^2\eta_r\sim \int|\sum_{k=1}^Ne(\frac{k}{N}x_1+\frac{k^2}{N^2}x_2)|^2\eta_r(x_1,x_2)\\
=\sum_{k=1}^N\sum_{k'=1}^N\widehat{\eta}_r(\xi(\frac{k-k'}{N},\frac{k^2-(k')^2}{N^2}))\sim  N\cdot N/r \cdot r^2=rN^2. 
\end{align*} 
Next we bound $\|F\|_{L^4(B_{R^\e r})}$ above. It follows from the local linear restriction statement (see \cite{demeterbook} Theorem 1.14, Prop 1.27, and Exercise 1.32) 
\[   \|f\|_{L^4(B_{R^\e r})}^4\lesssim C_\e R^{O(\e)} r^{-3}\|\widehat{f}\|_{L^4(\R^2)}^4
\]
that 
\begin{align*}
\|F\|_{L^4(B_{R^\e r})}^4&\sim \|\sum_{k=1}^Ne(\frac{k}{N}x_1+\frac{k^2}{N^2}x_2)\eta_r(x_1,x_2)\|_{L^4(B_{R^\e r})}^4\\
&\lesssim C_\e R^\e r^{-3}\|\sum_{k=1}^N\widehat{\eta}_r(\xi-(\frac{k}{N},\frac{k^2}{N^2}))\|_{L^4(\R^2)}^4. 
\end{align*}
The $L^4$ norm on the right hand side is bounded above by 
\begin{align*}
    \int_{B_2}|\sum_{k=1}^N\widehat{\eta}_r(\xi-(\frac{k}{N},\frac{k^2}{N^2}))|^4d\xi &\lesssim (Nr^{-1})^3\int_{B_2}\sum_{k=1}^N|\widehat{\eta}_r(\xi-(\frac{k}{N},\frac{k^2}{N^2}))|^4d\xi\\
    &\lesssim (Nr^{-1})^3(r^2)^3\int_{B_2}\sum_{k=1}^N|\widehat{\eta}_r(\xi-(\frac{k}{N},\frac{k^2}{N^2}))|d\xi\sim N^4r^3.
\end{align*}
This leads to the upper bound $\|F\|_{L^4(B_{R^\e r})}^4\lesssim (\log R)N^4$. 

Finally, by dyadic pigeonholing, there is some $\lambda\in [R^{-1000}, N]$ so that $\|F\|_{L^2(\eta_r)}^2\lesssim (\log R)\lambda^2|\{x\in B_{R^\e r}:|F(x)|\sim \lambda\}|+C_\e R^{-2000}$. The lower bound for $\|F\|_{L^2(\eta_r)}^2$ and the upper bound for $\|F\|_{L^4(B_{R^\e r})}^4$ tell us that
\begin{align*} 
\lambda^2rN^2\sim \lambda^2\|F\|_{L^2(\eta_r)}^2&\lesssim (\log R) \lambda^4|\{x\in B_{R^\e r}:|F(x)|\sim\lambda\}|+C_\e \lambda^4 R^{-2000}\\
&\lesssim (\log R)\|F\|_{L^4(B_{R^\e r})}^4+C_\e \lambda^4 R^{-2000}\lesssim C_\e R^\e N^4+C_\e \lambda^4 R^{-2000}. 
\end{align*} 
Conclude that $\lambda^2\lesssim C_\e R^\e N^2/r\sim C_\e R^\e\a^2$. Assuming $R$ is sufficiently large depending on $\e$,  
\[ rN^2\sim (\log R)\lambda^2|\{x\in B_{R^\e r}:|F(x)|\sim\lambda\}|\lesssim C_\e R^\e (N^2/r)|\{x\in B_{R^\e r}:|F(x)|\sim\lambda\}|,\]
so $|\{x\in B_{R^\e r}:|F(x)|\sim\lambda\}|\gtrsim C_\e^{-1}R^{-\e}r^2$ and $\lambda^2\gtrsim C_\e^{-1}R^{-\e}N^2/r\sim C_\e^{-1}R^{-\e}\a^2$. 

\noindent \underline{Case 2: $R^\b<\a^2\le R$.} Let $q$, $a$, and $b$ be integers satisfying
 \begin{equation}\label{cond} q\text{ odd},\quad 1\le b\le q\le N^{2/3}, \quad (b,q)=1,\quad \text{and}\quad 0\le a\le q.  \end{equation}
Define the set $M(q,a,b)$ to be
\[ M(q,a,b):= \{(x_1,x_2)\in[0,N]\times[0,N^2]:|x_1-\frac{a}{q}N|\leq \frac{1}{10^{10}},\quad |x_2-\frac{b}{q}N^2|\leq \frac{1}{10^{10} }\} . \]
\begin{lemma}\label{disjmaj} For each $(q,a,b)\not=(q',a',b')$, both tuples satisfying \eqref{cond}, $M(q,a,b)\cap M(q',a',b')=\emptyset$.
\end{lemma}
\begin{proof} If $\frac{b}{q}=\frac{b'}{q'}$, then using the relatively prime part of \eqref{cond}, $b=b'$ and $q=q'$. Then we must have $a\not=a'$, meaning that if $x_1$ is the first coordinate of a point in $M(q,a,b)\cap M(q,a',b')$, then 
\[ \frac{2}{10^{10}}\ge |x_1-\frac{a}{q}N|+|x_1-\frac{a'}{q}N|\ge \frac{|a-a'|N}{q}\ge N^{1/3}   \]
which is clearly a contradiction. The alternative is that $\frac{b}{q}\not=\frac{b'}{q'}$ in which case for $x_2$ the second coordinate of a point in $M(q,a,b)\cap M(q',a',b')$, 
\[  \frac{2}{10^{10}}\ge |x_2-\frac{b}{q}N^2|+|x_2-\frac{b'}{q'}N^2|\ge \frac{|b'q-bq'|N^2}{qq'}\ge \frac{N^2}{qq'}\ge N^{2/3},    \]
which is another contradiction. 
\end{proof}
\begin{lemma} \label{majest} For each $(x_1,x_2)\in M(q,a,b)$, $|F(x_1,x_2)|\sim \frac{N}{q^{1/2}}$, here meaning within a factor of $4$. 
\end{lemma}
\begin{proof}
This follows from Proposition 13.4 in \cite{demeterbook}.
\end{proof}
\begin{prop}\label{sharplow} Let $R^\b<\a^2\le R$ be given. There exists $v\in[0,N^2]$ satisfying  
\[ |\{(x_1,x_2)\in[0,R]^2:|F(x_1,x_2+v))|\ge \a\}|\gtrsim \frac{R^2N^3}{\a^6}  . \]
\end{prop}
\begin{proof} First note that by $N$-periodicity in $x_1$, 
\[|\{(x_1,x_2)\in[0,R]^2:|F(x_1,x_2+v))|\ge \a\}|\gtrsim \frac{R}{N}|\{(x_1,x_2)\in([0,N]\times[0,R]):|F(x_1,x_2+v))|\ge \a\}|. \]
The function $F$ is $N^2$ periodic in $x_2$, but $R<N^2$ so we need to find $v\in[0,N^2]$ making the set in the lower bound above largest. 

By Lemma \ref{majest}, it suffices to count the tuples $(q,a,b)$ satisfying \eqref{cond}, $q\le N^2/(16\a^2)$, and $|\frac{b}{q}N^2-v|\le R$, where $v$ is to be determined. Begin by considering the distribution of points $\frac{b}{q}$ in $[0,1]$, where $1\le b\le q\sim\frac{N^2}{\a^2}$, $(b,q)=1$. As in the proof of Lemma \ref{disjmaj}, if $\frac{b}{q}\not=\frac{b'}{q'}$, then $|\frac{b}{q}-\frac{b'}{q'}|\gtrsim \frac{\a^2}{N^4}$. Fix $b_0,q_0$ and consider the set $\{\frac{b}{q}:\frac{b}{q}=\frac{b_0}{q_0},\quad1\le b\le q\sim N^2/\a^2\}$. Let $q_m$ be maximal such that for some $1\le b_m\le q_m\sim N^2/\a^2$ and $(b_m,q_m)=1$, $\frac{b_m}{q_m}=\frac{b_0}{q_0}$. Then $q_0=q_m-k$ for some integer $k$ and $b_m(q_m-k)=b_0q_m$. Rearrange to get $q_m(1-\frac{b_0}{b_m})=k$.  Thus $q_0=q_m\frac{b_0}{b_m}\sim N^2/\a^2$, which implies that $\frac{b_0}{b_m}\sim 1$. Conclude that 
there are $\gtrsim \sum_{q\sim N^2/\a^2}\p(q)$ many unique points $\frac{b}{q}$ in $[0,1]$ satisfying our prescribed conditions for $\p$ denoting the Euler totient function. Use Theorem 3.7 in \cite{apostol} 
to estimate $\sum_{q\sim N^2/\a^2}\p(q)\sim N^4/\a^4$, as long as $N/\a$ is larger than some absolute constant. By the pigeonhole principle, there exists some $R/N^2$ interval $I\subset[0,1]$ containing $\sim \lceil\frac{ N^4}{\a^4}\frac{R}{N^2}\rceil$ many points $\frac{b}{q}$ with $1\le b\le q\sim N^2/\a^2$, $(b,q)=1$. There are also $\sim N^2/\a^2$ many choices for $a$ to complete the tuple $(q,a,b)$ satisfying \eqref{cond}. Let $c$ denote the center of $I$ and take $v=cN^2$ in the proposition statement and conclude that
\[ |\{(x_1,x_2)\in([0,N]\times[0,R]):|F(x_1,x_2+v))|\ge \a\}|\gtrsim \frac{RN^4}{\a^6}\]
to finish the proof. 
\end{proof} 
Note that Proposition \ref{sharplow} shows the sharpness of Theorem \ref{mainconj} in the range $R^\b<\a\le R$ since
\[ \frac{R^{2\b}}{\a^6}\sum_\g\|F_\g\|_2^2\sim \frac{R^{2\b}}{\a^6}R^\b R^2=\frac{N^3R^2}{\a^6}. \]
The sharpness of the trivial estimate $|U_\a\cap[-R,R]^2|\lesssim R^2$ in the range $\a^2<R^\b$ follows from Case 2 since for $\a^2<R^\b$, 
\[ |U_\a\cap[-R,R]^2|\ge |U_{R^{\b/2}}\cap[-R,R]^2|\gtrsim \frac{R^{2\b}}{(R^{\b/2})^6}\sum_\g\|F_\g\|_2^2\sim R^2. \]

\section{Implications of Theorem \ref{main}\label{implications}}

\begin{proof}[Proof of Theorem \ref{mainconj} from Theorem \ref{main}] First suppose that $\a^2> \frac{\lambda(1)^2}{\max_{s}\lambda(s^{-1}R^{-1})\lambda(s)} $. Then
\begin{align*}
    \max_{s}\lambda(s^{-1}R^{-1})\lambda(s)&\lesssim \max_{s}(s^{-1}R^{-1}R^\b)(sR^\b)=R^{2\b-1}\\
    &\le \begin{cases} R^{2\b-1}\quad&\text{if}\quad \a^2>R\\
\frac{R^{2\b}}{\a^2}\quad&\text{if}\quad R^\b\le \a^2\le R
\end{cases}. 
\end{align*}

Now suppose that 
$\a^2\le  \frac{\lambda(1)^2}{\max_{s}\lambda(s^{-1}R^{-1})\lambda(s)} $. Then
\begin{align*}
\frac{\lambda(1)^2}{\a^2}&\lesssim \begin{cases} R^{2\b-1}\quad&\text{if}\quad \a^2>R\\
\frac{R^{2\b}}{\a^2}\quad&\text{if}\quad R^\b\le \a^2\le R
\end{cases}. 
\end{align*}

\end{proof}

\begin{proof}[Proof of Corollary \ref{lqLp} from Theorem \ref{main}]
To see how this corollary follows from Theorem \ref{main}, first use an analogous series of pigeonholing steps as in Section 5 of \cite{gmw} to reduce to the case where  $\|f_\g\|_\infty\lesssim 1$ for all $\g$ and there exists $C>0$ so that $\|f_\g\|_p^p$ is either $0$ or comparable to $C$ for all $\g$. Split the integral
\[ \int|f|^p=\sum_{R^{-1000}\le \a\lesssim R^\b}\int_{U_\a}|f|^p+\int_{|f|<R^{-1000}}|f|^p \]
where $U_\a=\{x:|f(x)|\sim\a\}$ and assume via dyadic pigeonholing that
\[\int|f|^p\lesssim \a^p|U_\a| \]
(ignoring the case that the set where $|f|\le R^{-1000}$ dominates the integral which may be handled trivially). The result of all of the pigeonholing steps is that the statement of Corollary \ref{lqLp} follows from showing that
\[ \a^p|U_\a|\le C_\e R^\e (R^{\b p (1-\frac{1}{q})-(1+\b)}+R^{\b p(\frac{1}{2}-\frac{1}{q})})\lambda(1)^{\frac{p}{q}-1}\sum_\g\|f_\g\|_2^2 \]
where $f$ satisfies the hypotheses of Theorem \ref{main}. The full range $\frac{3}{p}+\frac{1}{q}\le 1$ follows from $p$ in the critical range $4\le p\le 6$, which we treat first. 
\newline\noindent{\underline{$4\le p\le 6$:}}
There are two cases depending on which upper bound is larger in Theorem \ref{main}. First we assume the $L^4$ bound holds, in which case
\begin{align*}
    \a^p|U_\a|&\le C_\e R^\e \a^{p-4}\underset{s}{\max}\lambda(s^{-1}R^{-1})\lambda(s)\sum_\g\|f_\g\|_2^2 \\
    &\sim C_\e R^\e \frac{\a^{p-4}}{\lambda(1)^{\frac{p}{q}-1}}\underset{s}{\max}\lambda(s^{-1}R^{-1})\lambda(s)(\sum_\g\|f_\g\|_p^q)^{\frac{p}{q}} \\
    &\lesssim C_\e R^\e \frac{\lambda(1)^{p-4}}{\lambda(1)^{\frac{p}{q}-1}}\underset{s}{\max}(R^\b s^{-1}R^{-1})(R^\b s)(\sum_\g\|f_\g\|_p^q)^{\frac{p}{q}}\\
    &\lesssim C_\e R^\e \lambda(1)^{p(1-\frac{1}{q})-3}R^{2\b-1}(\sum_\g\|f_\g\|_p^q)^{\frac{p}{q}} .
\end{align*}
Since $p(1-\frac{1}{q})-3\ge 0$, we may use the bound $\lambda(1)\lesssim R^\b$ to conclude that 
\[ \lambda(1)^{p(1-\frac{1}{q})-3}R^{2\b-1}\le R^{\b p(1-\frac{1}{q})-3\b+2\b-1}=R^{\b p(1-\frac{1}{q})-(1+\b)}.\]
The other case is that the $L^6$ bound holds in Theorem \ref{main}. We may also assume that $\a^2>\lambda(1)$ since otherwise we trivially have
\[ \a^p|U_\a|\le \lambda(1)^{ \frac{p}{2}-1}\sum_\g\|f_\g\|_2^2\sim \lambda(1)^{\frac{p}{2}-1+1-\frac{p}{q}}(\sum_\g\|f_\g\|_p^q)^{\frac{p}{q}}\lesssim R^{\b p(\frac{1}{2}-\frac{1}{q})}(\sum_\g\|f_\g\|_p^q)^{\frac{p}{q}}\]
where we used that $q\ge2$ since $4\le p\le 6$ and $\frac{3}{p}+\frac{1}{q}\le 1$. Now using the assumptions $\a^2>\lambda(1)$ and $p\le 6$, we have 
\begin{align*}
    \a^p|U_\a|&\le C_\e R^\e \a^{p-6}\lambda(1)^2\lambda(1)^{1-\frac{p}{q}}(\sum_\g\|f_\g\|_p^q)^{\frac{p}{q}} \\
    &\sim C_\e R^\e \lambda(1)^{p(\frac{1}{2}-\frac{1}{q})}(\sum_\g\|f_\g\|_p^q)^{\frac{p}{q}}\lesssim C_\e R^\e R^{\b p(\frac{1}{2}-\frac{1}{q})}(\sum_\g\|f_\g\|_p^q)^{\frac{p}{q}} .
\end{align*}
\newline\noindent{\underline{$3\le p< 4$:}} Suppose that $\a<R^{\b/2}$. Then using $L^2$-orthogonality, 
\[ \a^p|U_\a|\le R^{\frac{\b}{2}(p-2)}\sum_\g\|f_\g\|_2^2\sim R^{\frac{\b}{2}(p-2)}\lambda(1)^{1-\frac{p}{q}}(\sum_\g\|f_\g\|_p^q)^{\frac{p}{q}}. \]
Since in this subcase, $1-\frac{p}{q}\ge1-(p-3)> 0$, we are done after noting that $R^{\frac{\b}{2}(p-2)}\lambda(1)^{1-\frac{p}{q}}\le R^{\b p(\frac{1}{2}-\frac{1}{q})}$. Now assume that $\a\ge R^{\b/2}$ and use the $p=4$ case  above (noting that $R^{4\b(1-\frac{1}{q})-(1+\b)}\le R^{4\b(\frac{1}{2}-\frac{1}{q})}$) to get
\begin{align*}
    \a^p|U_\a|&\le \frac{\a^4}{(R^{\b/2})^{4-p}}|U_\a|\le R^{-\frac{\b}{2}(4-p)}C_\e R^\e R^{4\b (\frac{1}{2}-\frac{1}{q})}\lambda(1)^{\frac{4}{q}-1}\sum_\g\|f_\g\|_2^2 \\
    &\le C_\e R^\e R^{\b p(\frac{1}{2}-\frac{1}{q})}\lambda(1)^{\frac{p}{q}-1}\sum_\g\|f_\g\|_2^2 .
\end{align*}
\newline\noindent{\underline{$6<p$:}} In this range, we use the trivial bound $\a\le \lambda(1)$ and the $p=6$ case above (noting that $ R^{6\b(\frac{1}{2}-\frac{1}{q})}\le R^{6\b(1-\frac{1}{q})-(1+\b)}$) to get
\begin{align*}
    \a^p|U_\a|&\le \lambda(1)^{p-6}\a^6|U_\a| \le \lambda(1)^{p-6}C_\e R^\e R^{6\b(1-\frac{1}{q})-(1+\b)}\lambda(1)^{\frac{6}{q}-1}\sum_\g\|f_\g\|_2^2 \\
    &= \Big(\frac{\lambda(1)}{R^\b}\Big)^{(p-6)(1-\frac{1}{q})}C_\e R^\e R^{p\b(1-\frac{1}{q})-(1+\b)}\lambda(1)^{\frac{p}{q}-1}\sum_\g\|f_\g\|_2^2 \\
    &\le C_\e R^\e R^{p\b(1-\frac{1}{q})-(1+\b)}\lambda(1)^{\frac{p}{q}-1}\sum_\g\|f_\g\|_2^2.
\end{align*}

\end{proof}

\section{Tools to prove Theorem \ref{main} \label{tools}}

The proof of Theorem \ref{main} follows the high/low frequency decomposition and pruning approach from \cite{gmw}. In this section, we introduce notation for different scale neighborhoods of $\P^1$, a pruning process for wave packets at various scales, some high/low lemmas which are used to analyze the high/low frequency parts of square functions, and a version of a bilinear restriction theorem for $\P^1$.   

Begin by fixing some notation, as above. Let $\b\in[\frac{1}{2},1]$ and $R\ge2$. The parameter $\a>0$ describes the superlevel set 
\[ U_\a=\{x\in\R^2:|f(x)|\ge \a\}.\]
For $\e>0$, we analyze scales $R_k=R^{k\e}$
, noting that $R^{-1/2}\le R_k^{-1/2}\le 1$.
Let $N$ distinguish the index so that $R_N$ is closest to $R$. Since $R$ and $R_N$ differ at most by a factor of $R^\e$, we will ignore the distinction between $R_N$ and $R$ in the rest of the argument. 

Define the following collections, each of which partitions a neighborhood of $\P$ into approximate rectangles. 
\begin{enumerate}
    \item $\{\g\}$ is a partition of $\mc{N}_{R^{-1}}(\P^1)$ by approximate $R^{-\b}\times R^{-1}$ rectangles, described explicitly in \eqref{blocks}. 
    \item $\{\theta\}$ is a partition of $\mc{N}_{R^{-1}}(\P^1)$ by approximate $R^{-1/2}\times R^{-1}$ rectangles. In particular, let each $\theta$ be a union of adjacent $\g$. 
    \item $\{\tau_k\}$ is a partition of $\mc{N}_{R_k^{-1}}(\P^1)$ by approximate $R_k^{-1/2}\times R_k^{-1}$ rectangles. Assume the additional property that $\g\cap\tau_k=\emptyset$ or $\g\subset\tau_k$. 
\end{enumerate}

\subsection{A pruning step}

We will define wave packets at each scale $\tau_k$, and prune the wave packets associated to $f_{\tau_k}$ according to their amplitudes. 

For each $\tau_k$, fix a dual rectangle $\tau_k^*$ which is a $2R_k^{1/2}\times 2R_k$ rectangle centered at the origin and comparable to the convex set
\[ \{x\in\R^2:|x\cdot\xi|\le 1\quad\forall\xi\in\tau_k\}. \]
Let $\T_{\tau_k}$ be the collection of tubes $T_{\tau_k}$ which are dual to $\tau_k$, contain $\tau_k^*$, and which tile $\R^2$. Next, we will define an associated partition of unity $\s_{T_{\tau_k}}$. 
First let $\p(\xi)$ be a bump function supported in $[-\frac{1}{4},\frac{1}{4}]^2$. For each $m\in\Z^2$, let 
\[ \s_m(x)=c\int_{[-\frac{1}{2},\frac{1}{2}]^2}|\widecheck{\p}|^2(x-y-m)dy, \]
where $c$ is chosen so that $\sum_{m\in\Z^2}\s_m(x)=c\int_{\R^2}|\widecheck{\p}|^2=1$. Since $|\widecheck{\p}|$ is a rapidly decaying function, for any $n\in\N$, there exists $C_n>0$ such that
\[ \s_m(x)\le c\int_{[0,1]^2}\frac{C_n}{(1+|x-y-m|^2)^n}dy \le \frac{\tilde{C}_n}{(1+|x-m|^2)^n}. \]
Define the partition of unity $\s_{T_{\tau_k}}$ associated to ${\tau_k}$ to be $\s_{T_{\tau_k}}(x)=\s_m\circ A_{\tau_k}$, where $A_{\tau_k}$ is a linear transformation taking $\tau_k^*$ to $[-\frac{1}{2},\frac{1}{2}]^2$ and $A_{\tau_k}(T_{\tau_k})=m+[-\frac{1}{2},\frac{1}{2}]^2$. The important properties of $\s_{T_{\tau_k}}$ are (1) rapid decay off of $T_{\tau_k}$ and (2) Fourier support contained in $\tau_k$. 

To prove upper bounds for the size of $U_\a$, we will actually bound the sizes of $\sim \e^{-1}$ many subsets which will be denoted $U_\a\cap\Omega_k$, $U_\a\cap H$, and $U_\a\cap L$. The pruning process sorts between important and unimportant wave packets on each of these subsets, as described in Lemma \ref{ftofk} below. 

Partition $\T_{\theta}=\T_{\theta}^{g}\sqcup\T_{\theta}^{b}$ into a ``good" and a ``bad" set as follows. Let $\d>0$ be a parameter to be chosen in \textsection\ref{bilred} and set 
\[ T_{\theta}\in\T_{\theta}^{g}\quad\text{if}\quad \|\s_{T_{\theta}}f_{\theta}\|_{L^\infty(R^{2})}\le R^{M\d}\frac{\lambda(1)}{\a}  \]
where $M>0$ is a universal constant we will choose in the proof of Proposition \ref{mainprop}.

\begin{definition}[Pruning with respect to $\tau_k$]\label{taukprune} For each $\theta$ and $\tau_{N-1}$, define the notation $f_\theta^N=\sum_{T_\theta\in\T_{\theta}^{g}}\s_{T_\theta} f_\theta$ and $f_{\tau_{N-1}}^{N}=\sum_{\theta\subset\tau_{N-1}}f_\theta^N$. For each $k<N$, let 
\begin{align*} \T_{\tau_k}^{g}&=\{T_{\tau_k}\in\T_{\tau_{k}}:\|\s_{T_{\tau_{k}}}f_{\tau_{k}}^{k+1}\|_{L^\infty(R^2)}\le R^{M\d}\frac{\lambda(1)}{\a}\}, \\
f_{\tau_{k}}^{k}=\sum_{T_{\tau_k}\in\T_{\tau_k}^{g}}&\s_{T_{\tau_k}}f^{k+1}_{\tau_k}\qquad\text{and}\qquad  f_{\tau_{k-1}}^{k}=\sum_{\tau_k\subset\tau_{k-1}}f_{\tau_k}^k .
\end{align*}
\end{definition}
For each $k$, define the $k$th version of $f$ to be $f^k=\underset{\tau_k}{\sum}f_{\tau_k}^k$.
\begin{lemma}[Properties of $f^k$] \label{pruneprop}
\begin{enumerate} 
\item $ | f_{\tau_{k}}^k (x) | \le |f_{ \tau_{k}}^{k+1}(x)|\le \#\g\subset\tau_k. $
\item $ \| f_{\tau_k}^k \|_{L^\infty} \le C_\e R^{O(\e)}R^{M\d}\frac{\lambda(1)}{\a}$.
\item\label{item3} $ \text{supp} \widehat{f_{\tau_k}^k}\subset 2\tau_k . $
\item \label{item4} $  \text{supp} \widehat{f_{\tau_{k-1}}^k}\subset (1+(\log R)^{-1})\tau_{k-1}. $
\end{enumerate}
\end{lemma}
\begin{proof}
The first property follows because $\sum_{T_{\tau_k} \in \T_{\tau_k}}  \s_{T_{\tau_k}}$ is a partition of unity, and
$$ f_{\tau_k}^k=\sum_{T_{\tau_k}\in\T_{\tau_k^h}} \s_{T_{\tau_k}}f_{\tau_k}^{k+1}. $$
Furthermore, by definition of $f_{\tau_k}^{k+1}$ and iterating, we have
\begin{align*}
|f_{\tau_k}^k|\le |f_{\tau_k}^{k+1}|&\le \sum_{\tau_{k+1}\subset\tau_k}|f_{\tau_{k+1}}^{k+1}|\le\cdots\le \sum_{\tau_N\subset\tau_k}|f_{\tau_N}^N|\\
&\le \sum_{\theta\subset\tau_k}|f_\theta|\le \sum_{\g\subset\tau_k}|f_\g|\lesssim \#\g\subset\tau_k 
\end{align*}
where we used the assumption $\|f_\g\|_\infty\lesssim 1$ for all $\g$. 
Now consider the $L^\infty$ bound in the second property.  We write
$$ f_{ \tau_k}^k(x) = \sum_{\substack{T_{\tau_k} \in \T_{\tau_k^h},\\ x \in R^\e T_{\tau_k}}} \s_{T_{\tau_k}} f_{ \tau_k}^{k+1} + \sum_{\substack{T_{\tau_k} \in \T_{\tau_k, \lambda},\\ x \notin R^\e T_{\tau_k}}} \s_{T_{\tau_k}} f_{k+1, \tau_k}. $$

\noindent The first sum has at most $C R^{2\e}$ terms, and each term has norm bounded by $R^{M\d}\frac{\lambda(1)}{\a}$ by the definition of $\T_{\tau_k}^h$.  By the first property, we may trivially bound $f_{\tau_k}^{k+1}$ by $R\max_\g\|f_\g\|_\infty$. But if $x \notin R^\e T_{\tau_k}$, then $\s_{T_{\tau_k}}(x) \le R^{-1000}$. Thus 
\begin{align*} 
|\sum_{\substack{T_{\tau_k} \in \T_{\tau_k}^h,\\  x \notin R^\e T_{\tau_k}}} \s_{T_{\tau_k}} f^{k+1}_{ \tau_k}|&\le \sum_{\substack{T_{\tau_k} \in \T_{\tau_k}^h,\\ x \notin R^\e T_{\tau_k}}} R^{-500}\s_{T_{\tau_k}}^{1/2}(x) \|f^{k+1}_{ \tau_k}\|_\infty\le R^{-250}\max_\g\|f_\g\|_\infty. 
\end{align*} 
Since $\a\lesssim|f(x)|\lesssim \sum_\g\|f_\g\|_\infty\lesssim \lambda(1)$, (recalling the assumption that each $\|f_\g\|_\infty\lesssim 1$), we note $R^{-250}\le CR^{2\e}\frac{\lambda(1)}{\a}$. 

The fourth and fifth properties depend on the Fourier support of $\s_{T_{\tau_k}}$, which is contained in $\frac{1}{2}\tau_k$. Initiate a 2-step induction with base case $k=N$: $f_{\theta}^N$ has Fourier support in $2\theta$ because of the above definition. Then 
\[ f_{\tau_{N-1}}^N=\sum_{\theta\subset\tau_{N-1}}f_{\theta}^N \]
has Fourier support in $\underset{\theta\subset\tau_{N-1}}{\cup}2\theta$, which is contained in $(1+(\log R)^{-1})\tau_{N-1}$. Since each $\s_{T_{\tau_{N-1}}}$ has Fourier support in $\frac{1}{2}\tau_{N-1}$, 
\[ f_{\tau_{N-1}}^{N-1}=\sum_{T_{\tau_{N-1}}\in\T_{\tau_{N-1},\lambda}}\s_{\tau_{N-1}}f_{\tau_{N-1}}^N  \]
has Fourier support in $\frac{1}{2}\tau_{N-1}+(1+(\log R)^{-1})\tau_{N-1}\subset 2\tau_{N-1}$. Iterating this reasoning until $k=1$ gives \eqref{item3} and \eqref{item4}.

\end{proof}

\begin{definition} For each $\tau_k$, let $w_{\tau_k}$ be the weight function adapted to $\tau_k^*$ defined by 
\[ w_{\tau_k}(x)=w_k\circ R_{\tau_k}(x)\]
where 
\[ w_k(x,y)=\frac{c}{(1+\frac{|x|^2}{R_k})^{10}(1+\frac{|y|^2}{R_k^2})^{10}},\qquad \|w\|_1=1,\]
and $R_{\tau_k}:\R^2\to\R^2$ is the rotation taking $\tau_k^*$ to $[-R_k^{1/2},R_k^{1/2}]\times[-R_k,R_k]$. For each $T_{\tau_k}\in\T_{\tau_k}$, let $w_{T_{\tau_k}}=w_{\tau_k}(x-c_{T_{\tau_k}})$ where $c_{T_{\tau_k}}$ is the center of $T_{\tau_k}$. For $s>0$, we also use the notation $w_{s}$ to mean
\begin{equation}\label{ballweight} w_s(x)=\frac{c'}{(1+|x|^2/s^2)^{10}},\qquad\|w_s\|_1=1. \end{equation}
\end{definition}
The weights $w_{\tau_k}$, $w_{\theta}=w_{\tau_N}$, and $w_{s}$ are useful when we invoke the locally constant property. By locally constant property, we mean generally that if a function $f$ has Fourier transform supported in a convex set $A$, then for a bump function $\p_A\equiv 1$ on $A$, $f=f*\widecheck{\p_A}$. Since $|\widecheck{\p_A}|$ is an $L^1$-normalized function which is positive on a set dual to $A$, $|f|*|\widecheck{\p_A}|$ is an averaged version of $|f|$ over a dual set $A^*$. We record some of the specific locally constant properties we need in the following lemma.  
\begin{lemma}[Locally constant property]\label{locconst} For each $\tau_k$ and $T_{\tau_k}\in\T_{\tau_k}$, 
\[\|f_{\tau_k}\|_{L^\infty(T_{\tau_k})}^2\lesssim |f_{\tau_k}|^2*w_{\tau_k}(x)\qquad\text{for any}\quad x\in T_{\tau_k} .\]
For any collection of $\sim s^{-1}\times s^{-2}$ blocks $\theta_s$ partitioning $\mc{N}_{s^{-2}}(\P^1)$ and any $s$-ball $B$,
\[\|\sum_{\theta_s}|f_{\theta_s}|^2\|_{L^\infty(B)}\lesssim \sum_{\theta_s}|f_{\theta_s}|^2*w_{s}(x)\qquad\text{for any}\quad x\in B.  \]
\end{lemma}
Because the pruned versions of $f$ and $f_{\tau_k}$ have essentially the same Fourier supports as the unpruned versions, the locally constant lemma applies to the pruned versions as well. 
\begin{proof}[Proof of Lemma \ref{locconst}] Let $\rho_{\tau_k}$ be a bump function equal to $1$ on $\tau_k$ and supported in $2\tau_k$. Then using Fourier inversion and H\"{o}lder's inequality, 
\[ |f_{\tau_k}(y)|^2=|f_{\tau_k}*\widecheck{\rho_{\tau_k}}(y)|^2\le\|\widecheck{\rho_{\tau_k}}\|_1 |f_{\tau_k}|^2*|\widecheck{\rho_{\tau_k}}|(y). \]
Since $\rho_{\tau_k}$ may be taken to be an affine transformation of a standard bump function adapted to the unit ball, $\|\widecheck{\rho_{\tau_k}}\|_1$ is a constant. The function $\widecheck{\rho_{\tau_k}}$ decays rapidly off of $\tau_k^*$, so $|\widecheck{\rho_{\tau_k}}|\lesssim w_{{\tau_k}}$.
Since for any $T_{\tau_k}\in\T_{\tau_k}$, $w_{\tau_k}(y)$ is comparable for all $y\in T_{\tau_k}$, we have
\begin{align*} \sup_{x\in T_{\tau_k}}|f_{\tau_k}|^2*w_{\tau_k}(x)&\le \int|f_{\tau_k}|^2(y)\sup_{x\in T_{\tau_k}}w_{\tau_k}(x-y)dy\\
&\sim \int|f_{\tau_k}|^2(y)w_{\tau_k}(x-y)dy\qquad \text{for all}\quad x\in T_{\tau_k}. 
\end{align*}
For the second part of the lemma, repeat analogous steps as above, except begin with $\rho_{\theta_s}$ which is identically $1$ on a ball of radius $2s^{-1}$ containing $\theta_s$. Then 
\[  \sum_{\theta_s}|f_{\theta_s}(y)|^2=\sum_{\theta_s}|f_{\theta_s}*\widecheck{\rho_{\theta_s}}(y)|^2\lesssim \sum_{\theta_s}|f_{\theta_s}|^2*|\widecheck{\rho_{s^{-1}}}|(y),\]
where we used that each $\rho_{\theta_s}$ is a translate of a single function $\rho_{s^{-1}}$. The rest of the argument is analogous to the first part. 
\end{proof}

\begin{definition}[Auxiliary functions] Let $\p(x):\R^2\to[0,\infty)$ be a radial, smooth bump function satisfying $\p(x)=1$ on $B_1$ and $\supp\p\subset B_2$. 
\begin{align*}
    \p(2^{J+1}\xi)+\sum_{j=-2}^J[\p(2^j\xi)-\p(2^{j+1}\xi)]
\end{align*}
where $J$ is defined by $2^{J}\le \lceil R^\b\rceil< 2^{J+1}$. Then for each dyadic $s=2^j$, let
\[ \eta_{\sim s}(\xi) =\p(2^j\xi)-\p(2^{j+1}\xi) \]
and let
\[ \eta_{<\lceil R^\b\rceil^{-1}}(\xi)=\p(2^{J+1}\xi). \]
Finally, for $k=1,\ldots,N-1$, define
\[ \eta_k(\xi)=\p(R_{k+1}^{1/2}x).  \]
\end{definition}

\vspace{3mm}
\begin{definition} Let $G(x)=\sum_\theta|f_\theta|^2*w_\theta$, $G^\ell(x)=G*\widecheck{\eta}_{<\lceil R^{\b}\rceil^{-1}}$, $G^h(x)=G(x)-G^\ell(x)$. For $k=1,\ldots,N-1$, let
\[ g_k(x)=\sum_{\tau_k}|f_{\tau_k}^{k+1}|^2*w_{\tau_k},\qquad g_k^{\ell}(x)=g_k*\widecheck{\eta}_k,\qquad \text{and}\qquad g_k^h(x)=g_k-g_k^{\ell}.  \]
\end{definition}
\vspace{3mm}
\begin{definition} Define the high set
\[ H=\{x\in B_R:G(x)\le 2|G^h(x)|\}. \] 
For each $k=1,\ldots,N-1$, let 
\[ \Omega_k=\{x\in B_R\setminus H:g_k\le 2|g_k^h|,\,g_{k+1}\le 2|g_{k+1}^\ell|,\,\ldots,\,g_N\le2|g_N^\ell|\} \]
and for each $k=1,\ldots,N$. Define the low set
\[ L=\{x\in B_R\setminus H:g_1\le 2|g_{1}^\ell|,\,\ldots,\,g_N\le2|g_N^\ell|,G(x)\le 2|G^\ell(x)|\}. \]
\end{definition}
\vspace{3mm}

\subsection{High/low frequency lemmas }



\begin{lemma}[Low lemma]\label{low} For each $x$, $|G^\ell(x)|\lesssim \lambda(1)$ and $|g_k^\ell(x)|\lesssim g_{k+1}(x)$. 
\end{lemma}
\begin{proof} For each $\theta$, by Plancherel's theorem,
\begin{align}
|f_\theta|^2*\widecheck{\eta}_{<\lceil R^\b\rceil^{-1}}(x)&= \int_{\R^2}|f_\theta|^2(x-y)\widecheck{\eta}_{<\lceil R^\b\rceil^{-1}}(y)dy \nonumber \\
&=  \int_{\R^2}\widehat{f}_\theta*\widehat{\overline{f}}_\theta(\xi)e^{-2\pi i x\cdot\xi}\eta_{<\lceil R^\b\rceil^{-1}}(\xi)d\xi \nonumber \\
&=  \sum_{\g,\g'\subset\theta}\int_{\R^2}e^{-2\pi i x\cdot\xi}\widehat{f}_{\g}*\widehat{\overline{f}}_{\g'}(\xi)\eta_{<\lceil R^\b\rceil^{-1}}(\xi)d\xi .\label{dis2}\nonumber
\end{align}
The integrand is supported in $(\g\setminus\g')\cap B_{2\lceil R^\b\rceil^{-1}}$. This means that the integral vanishes unless $\g$ is within $CR^{-\b}$ of $\g'$ for some constant $C>0$, in which case we write $\g\sim\g'$. Then 
\[\sum_{\g,\g'\subset\theta}\int_{\R^2}e^{-2\pi i x\cdot\xi}\widehat{f}_{\g}*\widehat{\overline{f}}_{\g'}(\xi)\eta_{<\lceil R^\b\rceil^{-1}}(\xi)d\xi=\sum_{\substack{\g,\g'\subset\theta\\\g\sim\g'}}\int_{\R^2}e^{-2\pi i x\cdot\xi}\widehat{f}_{\g}*\widehat{\overline{f}}_{\g'}(\xi)\eta_{<\lceil R^\b\rceil^{-1}}(\xi)d\xi \]
Use Plancherel's theorem again to get back to a convolution in $x$ and conclude that
\begin{align*}
|G*\widecheck{\eta}_{<\lceil R^\b\rceil^{-1}}(x)|&=\Big|\sum_\theta \sum_{\substack{\g,\g'\subset\theta\\\g\sim\g'}}(f_\g\overline{f}_{\g'})*w_\theta*\widecheck{\eta}_{<\lceil R^\b\rceil^{-1}}(x) \Big|\\
&\lesssim \sum_\theta\sum_{\g\subset\theta}|f_\g|^2*w_\theta*|\widecheck{\eta}_{<\lceil R^\b\rceil^{-1}}|(x)\lesssim\sum_{\g}\|f_\g\|_\infty^2\lesssim\lambda(1). 
\end{align*}
By an analogous argument as above, we have that
\[ |g_k^\ell(x)|\lesssim \sum_{\tau_{k+1}}|f_{\tau_{k+1}}^{k+1}|^2*w_{\tau_k}*|\widecheck{\eta}_k|(x) \]
where for each summand, $w_{\tau_k}$ corresponds to the $\tau_k$ containing $\tau_{k+1}$. By definition, $|f_{\tau_{k+1}}^{k+1}|\le |f_{\tau_{k+1}}^k|$. By the locally constant property, $|f_{\tau_{k+1}}^k|^2\lesssim |f_{\tau_{k+1}}|^2*w_{\tau_{k+1}}$. It remains to note that
\[ w_{\tau_{k+1}}*w_{\tau_k}*|\widecheck{\eta}_k|(x)\lesssim w_{\tau_{k+1}}(x) \]
since $\tau_k^*\subset\tau_{k+1}^*$ and $\widecheck{\eta}_k$ is an $L^1$-normalized function that is rapidly decaying away from $B_{R_{k+1}^{1/2}}(0)$.

\end{proof}

\begin{lemma}[Pruning lemma]\label{ftofk} For any $\tau$, 
\begin{align*} 
|\sum_{\tau_k\subset\tau}f_{\tau_k}-\sum_{\tau_k\subset\tau}f_{\tau_k}^{k+1}(x)|&\le C_\e R^{-M\d}\a \qquad\text{for all $x\in \Omega_k$}\\
\text{and}\qquad |\sum_{\tau_1\subset\tau}f_{\tau_1}-\sum_{\tau_1\subset\tau}f_{\tau_1}^{1}(x)|&\le C_\e R^{-M\d}\a\qquad \text{ for all $x\in L$}. \end{align*}
\end{lemma}

\begin{proof} 

By the definition of the pruning process, we have 
\[ f_{\tau}=f^N_{\tau}+(f_{\tau}-f^N_{\tau})=\cdots=f^{k+1}_{\tau}(x)+\sum_{m=k+1}^N(f^{m+1}_{\tau}-f^{m}_{\tau})\]
with the understanding that $f^{N+1}=f$ and formally, the subscript $\tau$ means $f_\tau=\sum_{\g\subset\tau}f_\g$ and $f_{\tau}^m=\sum_{\tau_m\subset\tau}f_{\tau_m}^m$. We will show that each difference in the sum is much smaller than $\a$.
For each $m\ge k+1$ and $\tau_m$, 
\begin{align*}
    |f_{\tau_m}^m(x)-f_{\tau_m}^{m+1}(x)|&=|\sum_{T_{\tau_m}\in\T_{\tau_m}^{b}}\s_{T_{\tau_m}}(x)f_{\tau_m}^{m+1}(x)|  = \sum_{T_{\tau_m}\in T_{\tau_m}^b} |\s_{T_{\tau_m}}^{1/2}(x)f_{\tau_m}^{m+1}(x)|\s_{T_{\tau_m}}^{1/2}(x) \\
     & \lesssim\sum_{T_{\tau_m}\in \T_{\tau_m}^b}  R^{-M\d}\frac{\a}{\lambda(1)}\lambda^{-1} \| \s_{T_{\tau_m}}f_{{\tau_m}}^{m+1} \|_{L^\infty(\R^2)} \| \s_{T_{\tau_m}}^{1/2}f_{{\tau_m}}^{m+1} \|_{L^\infty(\R^2)}  \s_{T_{\tau_m}}^{1/2}(x) \\
     & \lesssim R^{-M\d}\frac{\a}{\lambda(1)} \sum_{T_{\tau_m}\in \T_{\tau_m}^b} \| \s_{T_{\tau_m}}^{1/2}f_{{\tau_m}}^{m+1} \|_{L^\infty(\R^2)}^2 \s_{T_{\tau_m}}^{1/2}(x) \\
     & \lesssim R^{-M\d}\frac{\a}{\lambda(1)}\sum_{T_{\tau_m}\in \T_{\tau_m}^b}
      \sum_{\tilde{T}_{{\tau_m}}} \| \s_{T_{\tau_m}}|f_{{\tau_m}}^{m+1}|^2 \|_{L^\infty(\tilde{T}_{{\tau_m}})} \s_{T_{\tau_m}}^{1/2}(x) \\
     & \lesssim R^{-M\d}\frac{\a}{\lambda(1)} \sum_{T_{\tau_m},\tilde{T}_{\tau_m}\in \T_{\tau_m}} \| \s_{T_{\tau_m}}\|_{L^\infty(\tilde{T}_{\tau_m})}\||f_{{\tau_m}}^{m+1} |^2\|_{{L}^\infty(\tilde{T}_{{\tau_m}})} \s_{T_{\tau_m}}^{1/2}(x) .
\end{align*}
Let $c_{\tilde{T}_{\tau_m}}$ denote the center of $\tilde{T}_{\tau_m}$ and note the pointwise inequality
\[ \sum_{{T}_{\tau_m}}\|\s_{T_{\tau_m}}\|_{L^\infty(\tilde{T}_{\tau_m})}\s_{T_{\tau_m}}^{1/2}(x)\lesssim R_m^{3/2}w_{\tau_m}(x-c_{\tilde{T}_{\tau_m}}) ,\]
which means that
\begin{align*}
|f_{\tau_m}^m(x)-f_{\tau_m}^{m+1}(x)| & \lesssim R^{-M\d}\frac{\a}{\lambda(1)} R_m^{3/2}\sum_{\tilde{T}_{\tau_m}\in T_{\tau_m}} w_{\tau_m}(x-c_{\tilde{T}_{\tau_m}})\||f_{{\tau_m}}^{m+1} |^2\|_{{L}^\infty(\tilde{T}_{{\tau_m}})} \\
&\lesssim R^{-M\d}\frac{\a}{\lambda(1)}R_m^{3/2} \sum_{\tilde{T}_{\tau_m}\in T_{\tau_m}} w_{\tau_m}(x-c_{\tilde{T}_{\tau_m}})|f_{{\tau_m}}^{m+1} |^2*w_{\tau_m}(c_{\tilde{T}_{\tau_m}})\\
&\lesssim R^{-M\d}\frac{\a}{\lambda(1)} |f_{{\tau_m}}^{m+1} |^2*w_{\tau_m}(x).
\end{align*}
where we used the locally constant property in the second to last inequality and the pointwise relation $w_{\tau_m}*w_{\tau_m}\lesssim w_{\tau_m}$
for the final inequality. Then 
\[
    |\sum_{\tau_m\subset\tau}f_{\tau_m}^m(x)-f_{\tau_m}^{m+1}(x)|\lesssim R^{-M\d}\frac{\a}{\lambda(1)}\sum_{\tau_m\subset\tau}|f_{\tau_m}^{m+1}|^2*w_{\tau_m}(x)\lesssim R^{-M\d}\frac{\a}{\lambda(1)}g_m(x). \]
By the definition of $\Omega_k$ and Lemma \ref{low}, $g_m(x)\le 2|g_m^\ell(x)|\le 2C g_{m+1}(x)\le\cdots\le (2C)^{\e^{-1}} G(x)\lesssim (2C)^{\e^{-1}}r$. 
Conclude that
\[|\sum_{\tau_m\subset\tau}f_{\tau_m}^m(x)-f_{\tau_m}^{m+1}(x)|\lesssim (2C)^{\e^{-1}}R^{-M\d}\a.\]

The claim for $L$ follows immediately from the above argument, using the low-dominance of $g_k$ for all $k$. 
\end{proof}

\begin{definition}
Call the distribution function  $\lambda$ associated to a function $f$ $(R,\e)$-normalized if for any $\tau_k,\tau_m$,  
\begin{align*}
    \#\{\tau_k\subset\tau_m:f_{\tau_k}\not=0\}&\le 100\frac{\lambda(R_m^{-1/2})}{\lambda(R_k^{-1/2})}. 
\end{align*}
\end{definition}

\begin{lemma}[High lemma I]\label{high} Assume that $f$ has an $(R,\e)$-normalized distribution function $\lambda(\cdot)$. For each dyadic $s$, $R^{-\b}\le s\le R^{-1/2}$,  
\[ \int_{\R^2}|G*\widecheck{\eta}_{\sim s}|^2\lesssim C_\e R^{2\e}\lambda(s^{-1}R^{-1})\lambda(s)\sum_\g\|f_\g\|_2^2.\]
\end{lemma}
\begin{proof} Organize the $\{\g\}$ into subcollections $\{\theta_s\}$ in which each $\theta_s$ is a union of $\g$ which intersect the same $\sim s$-arc of $\P^1$, where here for concreteness, $\sim s$ means within a factor of $2$. Then by Plancherel's theorem, since $\overline{\widecheck{\eta}}_{\sim s}={\widecheck{\eta}}_{\sim s}$, we have for each $\theta$
\begin{align}
|f_\theta|^2*\widecheck{\eta}_{\sim s}(x)&= \int_{\R^2}|f_\theta|^2(x-y)\widecheck{\eta}_{\sim s}(y)dy \nonumber \\
&=  \int_{\R^2}\widehat{f}_\theta*\widehat{\overline{f}}_\theta(\xi)e^{-2\pi i x\cdot\xi}\eta_{\sim s}(\xi)d\xi \nonumber \\
&=  \sum_{\theta_s,\theta_s'\subset\theta}\int_{\R^2}e^{-2\pi i x\cdot\xi}\widehat{f}_{\theta_s}*\widehat{\overline{f}}_{\theta_s'}(\xi)\eta_{\sim s}(\xi)d\xi .\label{dis}
\end{align}
The support of $\widehat{\overline{f}}_{\theta_s'}(\xi)=\int e^{-2\pi ix\cdot\xi}\overline{f}_{\theta_s'}(x)dx=\overline{\widehat{f}}_{\theta_s'}(-\xi)$ is contained in $-\theta_s'$. This means that the support of $\widehat{f}_{\theta_s}*\widehat{\overline{f}}_{\theta_s'}(\xi)$ is contained in $\theta_s-\theta_s'$. Since the support of $\eta_{\sim s}(\xi)$ is contained in the ball of radius $2s$, for each $\theta_s\subset\theta$, there are only finitely many $\theta_s'\subset\theta$ so that the integral in \eqref{dis} is nonzero. Thus we may write  
\[ G*\widecheck{\eta}_{\sim s}(x)= \sum_\theta|f_\theta|^2*w_\theta*\widecheck{\eta}_{\sim s}(x)=\sum_\theta \sum_{\substack{\theta_s,\theta_s'\subset\theta\\ \theta_s\sim\theta_s'}}(f_{\theta_s}\overline{f}_{\theta_s'})*w_\theta*\widecheck{\eta}_{\sim s}(x). \]
where the second sum is over $\theta_s,\theta_s'\subset\theta$ with $\text{dist}(\theta_s,\theta_s')<2s$. Using the above pointwise expression and then Plancherel's theorem, we have 
\begin{align*}
\int_{\R^2}|G*\widecheck{\eta}_{\sim s}|^2&=    \int_{\R^2}|\sum_\theta \sum_{\substack{\theta_s,\theta_s'\subset \theta\\ \theta_s\sim\theta_s'}}(f_{\theta_s}\overline{f}_{\theta_s'})*w_\theta*\widecheck{\eta}_{\sim s}|^2 \\
    &=    \int_{\R^2}|\sum_\theta \sum_{\substack{\theta_s,\theta_s'\subset \theta\\ \theta_s\sim\theta_s'}}(\widehat{f_{\theta_s}}*\widehat{\overline{f}}_{\theta_s'})\widehat{w}_\theta{\eta}_{\sim s}|^2 
\end{align*}
For each $\theta$, $\sum_{\substack{\theta_s,\theta_s'\subset \theta\\ \theta_s\sim\theta_s'}}(\widehat{f_{\theta_s}}*\widehat{\overline{f}}_{\theta_s'})$ is supported in $\theta-\theta$, since each summand is supported in $\theta_s-\theta_s'$ and $\theta_s,\theta_s'\subset\theta$. For each $\xi\in\R^2$, $|\xi|>\frac{1}{2}r$, the maximum number of $\theta-\theta$ containing $\xi$ is bounded by the maximum number of $\theta$ intersecting an $R^{-1/2}\cdot s^{-1}R^{-1/2}$-arc of the parabola. Using that $\lambda(\cdot)$ is $(R,\e)$-normalized, this number is bounded above by  $C_\e R^\e \frac{\lambda(s^{-1}R^{-1})}{\lambda(R^{-1/2})}$. Since $\eta_{\sim s}$ is supported in the region $|\xi|>\frac{1}{2}r$, by Cauchy-Schwarz
\begin{align*}
\int_{\R^2}|\sum_\theta \sum_{\substack{\theta_s,\theta_s'\subset \theta\\ \theta_s\sim\theta_s'}}(\widehat{f_{\theta_s}}*\widehat{\overline{f}}_{\theta_s'})\widehat{w}_\theta{\eta}_{\sim s}|^2 &\lesssim C_\e R^\e\frac{\lambda(r^{-1}R^{-1})}{\lambda(R^{-1/2})}\sum_\theta\int_{\R^2}| \sum_{\substack{\theta_s,\theta_s'\subset \theta\\ \theta_s\sim\theta_s'}}(\widehat{f_{\theta_s}}*\widehat{\overline{f}}_{\theta_s'})\widehat{w}_\theta{\eta}_{\sim s}|^2 \\
&=C_\e R^\e\frac{\lambda(r^{-1}R^{-1})}{\lambda(R^{-1/2})}\sum_\theta\int_{\R^2}| \sum_{\substack{\theta_s,\theta_s'\subset \theta\\ \theta_s\sim\theta_s'}}({f_{\theta_s}}{\overline{f}}_{\theta_s'})*w_\theta*\widecheck{\eta}_{\sim s}|^2 \\
&\lesssim C_\e R^\e\frac{\lambda(r^{-1}R^{-1})}{\lambda(R^{-1/2})}\sum_{\theta}\int_{\R^2}| \sum_{\theta_s\subset\theta}|f_{\theta_s}|^2*w_\theta*|\widecheck{\eta}_{\sim s}||^2. 
\end{align*}
It remains to analyze each of the integrals above:
\begin{align*}
\int_{\R^2}|\sum_{\theta_s\subset\theta}|f_{\theta_s}|^2*w_{\theta}*|\widecheck{\eta}_{\sim s}||^2    &\lesssim \|\sum_{\theta_s\subset\theta}|f_{\theta_s}|^2*w_{\theta}*|\widecheck{\eta}_{\sim s}|\|_\infty\int_{\R^2}\sum_{\theta_s\subset\theta }|f_{\theta_s}|^2*w_{\theta}*|\widecheck{\eta}_{\sim s}| .
\end{align*}
Bound the $L^\infty$ norms using the assumption that $\|f_\g\|_\infty\lesssim 1$ for all $\g$:
\begin{align*} 
\|\sum_{\theta_s\subset\theta }|f_{\theta_s}|^2*w_{\theta}*|\widecheck{\eta}_{\sim s}|\|_\infty\lesssim \sum_{\theta_s\subset\theta }\|f_{\theta_s}\|_\infty^2\lesssim \sum_{\theta_s\subset\theta }\|\sum_{\g\subset\theta_s}|f_{\g}|\|_\infty^2&\lesssim \lambda(R^{-1/2})\lambda(s).
\end{align*}
Finally, using Young's convolution inequality and the $L^2$-orthogonality of the $f_{\g}$, we have
\[ \int_{\R^2}\sum_{\theta_s\subset\theta}|f_{\theta_s}|^2*w_{\theta}*|\widecheck{\eta}_{\sim s}|\lesssim \int_{\R^2}\sum_{\theta_s\subset\theta}|f_{\theta_s}|^2= \sum_{\g\subset\theta}\|f_\g\|_2^2.\]
\end{proof}

\begin{lemma}[High lemma II]\label{high2} For each $k$,
\[ \int_{\R^2}|g_k^h|^2\lesssim R^{3\e} \sum_{\tau_k}\int_{\R^2}|f_{\tau_{k+1}}^{k+1}|^4.\]
\end{lemma}
\begin{proof} By Plancherel's theorem, we have 
\begin{align*}
\int_{\R^2}|g_k^h|^2&= \int_{\R^2}|g_k-g_k^\ell|^2\\ 
&=    \int_{\R^2}|\sum_{\tau_k}(\widehat{f_{\tau_k}^{k+1}}*\widehat{\overline{f_{\tau_k}^{k+1}}})\widehat{w}_{\tau_k}-\sum_{\tau_k}(\widehat{f_{\tau_k}^{k+1}}*\widehat{\overline{f^{k+1}_{\tau_k}}})\widehat{w}_{\tau_k}\eta_k|^2\\
&\le \int_{|\xi|>cR_{k+1}^{-1/2}}|\sum_{\tau_k}(\widehat{f_{\tau_k}^{k+1}}*\widehat{\overline{f_{\tau_k}^{k+1}}})\widehat{w}_{\tau_k}|^2
\end{align*}
since $(1-\eta_k)$ is supported in the region $|\xi|>cR_{k+1}^{-1/2}$ for some constant $c>0$. For each $\tau_k$, $\widehat{f_{\tau_k}^{k+1}}*\widehat{\overline{f_{\tau_k}^{k+1}}}$ is supported in $2\tau_k-2\tau_k$, using property \eqref{item4} of Lemma \ref{pruneprop}. The maximum overlap of the sets $\{2\tau_k-2\tau_k\}$ in the region $|\xi|\ge cR_{k+1}^{-1/2}$ is bounded by $\sim\frac{R_{k}^{-1/2}}{R_{k+1}^{-1/2}}\lesssim R^{\e}$. Thus using Cauchy-Schwarz, 
\begin{align*} 
\int_{|\xi|>cR_{k+1}^{-1/2}}|\sum_{\tau_k}(\widehat{f_{\tau_k}^{k+1}}*\widehat{\overline{f_{\tau_k}^{k+1}}})\widehat{w}_{\tau_k}|^2&\lesssim R^\e \sum_{\tau_k} \int_{|\xi|>cR_{k+1}^{-1/2}}|(\widehat{f_{\tau_k}^{k+1}}*\widehat{\overline{f_{\tau_k}^{k+1}}})\widehat{w}_{\tau_k}|^2\\
&\le R^\e \sum_{\tau_k} \int_{\R^2}|(\widehat{f_{\tau_k}^{k+1}}*\widehat{\overline{f_{\tau_k}^{k+1}}})\widehat{w}_{\tau_k}|^2\\
&= R^\e \sum_{\tau_k} \int_{\R^2}||f_{\tau_k}^{k+1}|^2*{w}_{\tau_k}|^2\le R^{3\e} \sum_{\tau_{k+1}} \int_{\R^2}|f_{\tau_{k+1}}^{k+1}|^4  
\end{align*} 
where we used Young's inequality with $\|w_{\tau_k}\|_1\lesssim 1$ and $f_{\tau_k}^{k+1}=\sum_{\tau_{k+1}\subset\tau_k}f_{\tau_{k+1}}^{k+1}$ with Cauchy-Schwarz again in the last line. 
\end{proof}

\subsection{Bilinear restriction} We will use the following version of a local bilinear restriction theorem, which follows from a standard C\'{o}rdoba argument \cite{cordoba1977kakeya} included here for completeness.
\begin{theorem}\label{bilrest} Let $S\ge 4$, $\frac{1}{2}\ge D\ge S^{-1/2}$, and $X\subset\R^2$ be any Lebesgue measurable set. Suppose that $\tau$ and $\tau'$ are $D$-separated subsets of $\mc{N}_{S^{-1}}(\P^1)$. Then for a partition $\{\theta_S\}$ of $\mc{N}_{S^{-1}}(\P^1)$ into $\sim S^{-1/2}\times S^{-1}$-blocks, we have
\[\int_{X}|f_{\tau}|^2(x)|f_{\tau'}|^2(x)dx\lesssim D^{-2}
\int_{\mc{N}_{S^{1/2}}(X)}|\sum_{\theta_S}|f_{\theta_S}|^2*w_{S^{1/2}}(x)|^2dx .\]
\end{theorem}
In the following proof, the exact definition of the $\sim S^{-1}\times S^{-1}$ blocks $\theta_S$ is not important. However, by $f_\tau$ and $f_{\tau'}$, we mean more formally $f_\tau=\sum_{\theta_S\cap\tau\not=\emptyset}f_{\theta_S}$ and $f_{\tau'}=\sum_{\theta_S\cap\tau'\not=\emptyset}f_{\theta_S}$. 
\begin{proof} Let $B$ be a ball of radius $S^{1/2}$ centered at a point in $X$. Let $\p_B$ be a smooth function satisfying $\p_B\gtrsim 1$ in $B$, $\p_B$ decays rapidly away from $B$, and $\widehat{\p_B}$ is supported in the $S^{-1/2}$ neighborhood of the origin. Then 
\begin{align*}
    \int_{X\cap B}|f_\tau|^2|f_{\tau'}|^2&\lesssim \int_{\R^2}|f_\tau|^2|f_{\tau'}|^2\p_B. 
\end{align*}
Since $S$ is a fixed parameter and $\theta_S$ are fixed $\sim S^{-1/2}\times S^{-1}$ blocks, simplify notation by dropping the $S$. Expand the squared terms in the integral above to obtain
\[ \int_{\R^2}|f_\tau|^2|f_{\tau'}|^2\p_B=\sum_{\substack{\theta_i\cap\tau\not=\emptyset\\\theta'_i\cap\tau'\not=\emptyset}}\int_{\R^2}f_{\theta_1}\overline{f}_{\theta_2}f_{\theta_{2}'}\overline{f}_{\theta_{1}'}\p_B. \]
By Placherel's theorem, each integral vanishes unless
\begin{equation}\label{supprel}(\theta_{1}-\theta_{2})\cap \mc{N}_{S^{-1/2}}(\theta_{1}'-\theta_{2}')\not=\emptyset . \end{equation}
Next we check that the number of tuples $(\theta_1,\theta_2,\theta_1',\theta_2')$ (with $\theta_1,\theta_2$ having nonempty intersection with $\tau$ and $\theta_1',\theta_2'$ having nonempty intersection with $\tau'$) satisfying \eqref{supprel} is $O(D^{-1})$. Indeed, suppose that $\xi<\xi'<\xi''<\xi'''$ satisfy 
\[ (\xi,\xi^2)\in\theta_{1},\quad(\xi',(\xi')^2)\in\theta_{2},\quad (\xi'',(\xi'')^2)\in\theta_{1}',\quad (\xi''',(\xi''')^2)\in\theta_{2}'  \]
and
\[ \xi-\xi'=\xi''-\xi'''+O(S^{-1/2}). \]
Then by the mean value theorem, $\xi^2-(\xi')^2=2\xi_1(\xi-\xi')$ for some $\xi<\xi_1<\xi'$ and $(\xi'')^2-(\xi''')^2=2\xi_2(\xi''-\xi''')$ for some $\xi''<\xi_2<\xi'''$. Since $(\xi_1,\xi_1^2)\in\tau$ and $(\xi_2,\xi_2^2)\in\tau'$, we also know that $|\xi_1-\xi_2|\ge D$. Putting everything together, we have
\begin{align*}
    |\xi^2-(\xi')^2-((\xi'')^2-(\xi''')^2)|&= 2 |\xi_1(\xi-\xi')-\xi_2(\xi''-\xi''')|\\
    &\ge 2|\xi_1-\xi_2||\xi-\xi'|-cS^{-1/2}\ge (2C-c) S^{-1/2}
\end{align*}
if either $\text{dist}((\xi,\xi^2),(\xi',(\xi')^2))$ or $\text{dist}((\xi'',(\xi'')^2),(\xi''',(\xi''')^2))$ is larger than $CD^{-1}S^{-1/2}$. Thus for a suitably large $C$, the heights will have difference larger than the allowed $O(S^{-1/2})$-neighborhood imposed by \eqref{supprel}. The conclusion is that 
\begin{align*}
    \sum_{\substack{\theta_i\cap\tau\not=\emptyset\\\theta'_i\cap\tau'\not=\emptyset}}\int_{\R^2}f_{\theta_1}\overline{f}_{\theta_2}f_{\theta_{2}'}\overline{f}_{\theta_{1}'}\p_B&=\sum_{\substack{\theta_1\cap\tau\not=\emptyset\\\theta'_1\cap\tau'\not=\emptyset}}\,\,\sum_{\substack{d(\theta_1,\theta_2)\le CD^{-1}S^{-1/2}\\ d(\theta_1',\theta_2')\le CD^{-1}S^{-1/2}}}\int_{\R^2}f_{\theta_1}\overline{f}_{\theta_2}f_{\theta_{2}'}\overline{f}_{\theta_{1}'}\p_B\\
    &\lesssim D^{-2} 
    \int_{\R^2}(\sum_\theta|f_\theta|^2)^2\p_B . 
\end{align*}
Using the locally constant property and summing over a finitely overlapping cover of $\R^2$ by $S^{1/2}-$balls $B'$ with centers $c_{B'}$, we have
\begin{align*}
\int_{\R^2}(\sum_\theta|f_\theta|^2)^2\p_B&\le \sum_{B'}|B|\|\sum_\theta|f_\theta|^2\|_{L^\infty(B')}^2\|\p_B\|_{L^\infty(B')}\\
&\le |B| \Big(\sum_{B'}\|\sum_\theta|f_\theta|^2\|_{L^\infty(B')}\|\p_B^{1/2}\|_{L^\infty(B')}\Big)^2\\
&\lesssim |B| \Big(\sum_{B'}\sum_\theta|f_\theta|^2*w_{S^{1/2}}(c_{B'})\|\p_B^{1/2}\|_{L^\infty(B')}\Big)^2\\
&\lesssim |B|^{-1} \Big(\int_{\R^2}\sum_\theta|f_\theta|^2*w_{S^{1/2}}(y)\p_B^{1/2}(y)dy\Big)^2\\
&\lesssim |B|^{-1} \Big(\int_{B}\sum_\theta|f_\theta|^2*w_{S^{1/2}}(y)dy\Big)^2\\
&\le \int_B\Big(\sum_\theta|f_\theta|^2*w_{S^{1/2}}\Big)^2
\end{align*}
where we used that $w_{S^{1/2}}*\p_B^{1/2}(y)\lesssim w_{S^{1/2}}*\chi_B(y)$ in the second to last line.

\end{proof}

\section{Proof of Theorem \ref{main}\label{pf}}

Theorem \ref{main} follows from the following proposition and a broad-narrow argument in \textsection\ref{bilred}. First we prove a version of Theorem \ref{main} where $U_\a$ is replaced by a ``broad" version of $U_\a$. 
\subsection{The broad version of Theorem \ref{main}}
Let $\delta>0$ be a parameter we will choose in the broad/narrow analysis. The notation $\ell(\tau)=s$ means that $\tau$ is an approximate $s\times s^2$ block which is part of a partition of $\mc{N}_{s^2}(\P^1)$. For two non-adjacent blocks $\tau,\tau'$ satisfying $\ell(\tau)=\ell(\tau')=R^{-\d}$, define the broad version of $U_\a$ to be
\begin{equation}\label{broad2} \text{Br}_\a(\tau,\tau')=\{x\in\R^2:\a\sim|f_{\tau}(x)f_{\tau'}(x)|^{1/2},\,\,(|f_\tau(x)|+|f_{\tau'}(x)|)\le R^{O(\d)}\a \}
\}. \end{equation}

\begin{proposition}\label{mainprop} Suppose that $f$ satisfies the hypotheses of Theorem \ref{main} and has an $(R,\e)$-normalized distribution function $\lambda(\cdot)$. Then
\begin{align*}
 |\text{\emph{Br}}_{\a}(\tau,\tau')|
 &\le C_{\e,\d} R^\e R^{O(\delta)}\begin{cases} \frac{1}{\a^4}\underset{s}{\max}\lambda(s^{-1}R^{-1})\lambda(s)\sum_\g\|f_\g\|_2^2\quad&\text{if}\quad \a^2>\frac{\lambda(1)^2}{\underset{s}{\max}\lambda(s^{-1}R^{-1})\lambda(s)}\\
\frac{\lambda(1)^2}{\a^6}\sum_\g\|f_\g\|_2^2\quad&\text{if}\quad  \a^2\le \frac{\lambda(1)^2}{\max_{s}\lambda(s^{-1}R^{-1})\lambda(s)} 
\end{cases}.
\end{align*}
\end{proposition}
\begin{proof}[Proof of Proposition \ref{mainprop}] \noindent{\underline{Bounding $|\text{Br}_\a(\tau,\tau')\cap H|$}:} Using bilinear restriction, given here by Theorem \ref{bilrest}, we have
\[\a^4|\text{Br}_\a(\tau,\tau')\cap H|\lesssim \sum_{\substack{\ell(\tau)=\ell(\tau)=R^{-\delta}\\d(\tau,\tau')\gtrsim R^{-\delta}}} \int_{U_\a\cap H} |f_{\tau}|^2|f_{\tau'}|^2\lesssim R^{O(\d)}\int_{\mc{N}_{R^{1/2}}(\text{Br}_\a(\tau,\tau')\cap H)}(\sum_\theta|f_\theta|^2*w_{R^{1/2}})^2 .\]
By the locally constant property and the pointwise inequality $w_{R^{1/2}}*w_{\theta}\lesssim w_\theta$ for each $\theta$, we have that $\sum_\theta|f_\theta|^2*w_{R^{1/2}}\lesssim G(x)$. Then 
\begin{align}
\int_{\mc{N}_{R^{1/2}}(\text{Br}_\a(\tau,\tau')\cap H)}|G(x)|^2dx \le \sum_{\substack{Q_{R^{1/2}}:\\Q_{R^{1/2}}\cap (\text{Br}_\a(\tau,\tau')\cap H)\not=\emptyset}}|Q_{R^{1/2}}|\|G\|_{L^\infty(Q_{R^{1/2}}\cap(\text{Br}_\a(\tau,\tau')\cap H))}^2\label{upp}
\end{align} 
For each $x\in H$, $G(x)\le 2|G^h(x)|$. Also note the equality $G^h(x)=\sum_{s}G*\widecheck{\eta}_{\sim s}(x)$ where the sum is over dyadic $s$ in the range $\lceil{R^\b}\rceil^{-1}\lesssim s\lesssim R^{-1/2}$. This is because the Fourier support of $G^h$ is contained in $\cup_\theta (\theta-\theta)\setminus B_{c\lceil R^{\b}\rceil^{-1}}$ for a sufficiently small $c>0$. By dyadic pigeonholing, there is some dyadic $s$, $\lceil{R^\b}\rceil^{-1}\lesssim s\lesssim R^{-1/2}$, so that the upper bound in \eqref{upp} is bounded by
\[ (\log R) \sum_{\substack{Q_{R^{1/2}}:\\Q_{R^{1/2}}\cap (\text{Br}_\a(\tau,\tau')\cap H)\not=\emptyset}}|Q_{R^{1/2}}|\|G*\widecheck{\eta}_{\sim s}\|_{L^\infty(Q_{R^{1/2}}\cap(\text{Br}_\a(\tau,\tau')\cap H))}^2. \]
By the locally constant property, the above displayed expression is bounded by 
\[ (\log R)\sum_{\substack{Q_{R^{1/2}}:\\Q_{R^{1/2}}\cap(\text{Br}_\a(\tau,\tau')\cap H)}}\int_{\R^2}|G*\widecheck{\eta}_{\sim s}|^2w_{Q_{R^{1/2}}} \lesssim (\log R)\int_{\R^2}|G*\widecheck{\eta}_{\sim s}|^2. \]
Use Lemma \ref{high} to upper bound the above integral to finish bounding $|\text{Br}_\a(\tau,\tau')\cap H|$.

\vspace{3mm}
\noindent{\underline{Bounding $|\text{Br}_\a(\tau,\tau')\cap\Omega_k|$}:} First write the trivial inequality 
\[\a^4|\text{Br}_\a(\tau,\tau')\cap\Omega_k|\le \sum_{\substack{\ell(\tau)=\ell(\tau)=R^{-\delta}\\d(\tau,\tau')\gtrsim R^{-\delta}}}\int_{\text{Br}_\a(\tau,\tau')\cap\Omega_k\cap\{|f_\tau f_{\tau'}|^{1/2}\sim \a\}}|f_{\tau}|^2|f_{\tau'}|^2 .\]
By the definition of $\text{Br}_\a(\tau,\tau')\cap\Omega_k$ and Lemma \ref{ftofk}, for each $x\in\text{Br}_\a(\tau,\tau')\cap\Omega_k$ we have 
\begin{align*}
|f_\tau(x)f_{\tau'}(x)|&\le |f_\tau(x)||f_{\tau'}(x)-f_{\tau'}^{k+1}(x)|+|f_\tau(x)-f_{\tau}^{k+1}(x)||f_{\tau'}^{k+1}(x)|+|f_{\tau}^{k+1}(x)f_{\tau'}^{k+1}(x)|\\
    &\lesssim C_\e R^{O(\d)}R^{-M\d}\a^2+|f_{\tau}^{k+1}(x)f_{\tau'}^{k+1}(x)|. 
\end{align*}
For $M$ large enough in the definition of pruning (depending on the implicit universal constant from the broad/narrow analysis which determines the set $\text{Br}_\a(\tau,\tau')$) so that $R^{O(\d)}R^{-M\d}\le R^{-\d}$ and for $R$ large enough depending on $\e$ and $\d$, we may bound each integral by
\[\int_{\{\text{Br}_\a(\tau,\tau')\cap\Omega_k\cap\{|f_\tau f_{\tau'}|^{1/2}\sim \a\}}|f_{\tau}|^2|f_{\tau'}|^2 \lesssim\int_{\text{Br}_\a(\tau,\tau')\cap\Omega_k}|f_{\tau}^{k+1}|^2|f_{\tau'}^{k+1}|^2. \]
Repeat analogous bilinear restriction, high-dominated from the definition of $\Omega_k$, and locally-constant steps from the argument bounding $\text{Br}_\a(\tau,\tau')\cap H$ to obtain
\begin{align*}
    \a^4|\text{Br}_\a(\tau,\tau')\cap\Omega_k|&\lesssim  R^{O(\d)} \int_{\R^2}|g_k^h|^2. 
\end{align*}
Use Lemma \ref{high2} and Lemma \ref{pruneprop} to bound the above integral, obtaining
\begin{align*}
    \a^4|\text{Br}_\a(\tau,\tau')\cap\Omega_k|&\lesssim  (\log R)^4 \int_{\R^2}|g_k^h|^2\\
    &\lesssim R^{O(\d)}R^{O(\e)}\frac{\lambda(1)^2}{\a^2}\sum_{\tau_{k+1}}\int_{\R^2}|f_{\tau_{k+1}}^{k+1}|^2. 
\end{align*}
Use $L^2$-orthogonality and that $|f_{\tau_m}^m|\le|f_{\tau_{m}}^{m+1}|$ for each $m$ to bound each integral above:
\begin{align*}
\int_{\R^2}|f_{\tau_{k+1}}^{k+1}|^2&\le \int_{\R^2}|f_{\tau_{k+1}}^{k+2}|^2 \le C\sum_{\tau_{k+2}\subset\tau_{k+1}}\int_{\R^2}|f_{\tau_{k+2}}^{k+2}|^2 \le\cdots\le C^{\e^{-1}}\sum_{\g\subset\tau_{k+1}}\int_{\R^2}|f_\g|^2. 
\end{align*} 
We are done with this case because
\begin{align*}
    \frac{\lambda(1)^2}{\a^2}&\le \begin{cases} \underset{s}{\max}\lambda(s^{-1}R^{-1})\lambda(s)\quad&\text{if}\quad \a^2>\frac{\lambda(1)^2}{\underset{s}{max}\lambda(s^{-1}R^{-1})\lambda(s)}\\
\frac{\lambda(1)^2}{\a^2}\quad&\text{if}\quad  \a^2\le \frac{\lambda(1)^2}{\max_{s}\lambda(s^{-1}R^{-1})\lambda(s)} 
\end{cases}.
\end{align*}

\noindent{\underline{Bounding $|\text{Br}_\a(\tau,\tau')\cap L|$}:} Repeat the pruning step from the previous case to get
\[ \a^6|\text{Br}_\a(\tau,\tau')\cap L|\lesssim\sum_{\substack{\ell(\tau)=\ell(\tau)=R^{-\delta}\\d(\tau,\tau')\gtrsim R^{-\delta}}} \int_{\text{Br}_\a(\tau,\tau')\cap L\cap\{|f_\tau f_{\tau'}|^{1/2}\sim\a\}}|f_{\tau}^1f_{\tau'}^1|^2|f_\tau f_{\tau'}| . \]
Use Cauchy-Schwartz and the locally constant lemma for the bound $|f_{\tau}^1f_{\tau'}^1|\lesssim R^{O(\e)}G_0$ and recall that by Lemma \ref{low}, $G_0\le C_\e R^\e \lambda(1)$. Then 
\begin{align*}
R^{O(\e)}\sum_{\substack{\ell(\tau)=\ell(\tau)=R^{-\delta}\\d(\tau,\tau')\gtrsim R^{-\delta}}}\int_{\text{Br}_\a(\tau,\tau')\cap L}|G_0|^2|f_\tau f_{\tau'}|&\le R^{O(\e)}\lambda(1)^2\sum_{\ell(\tau)=R^{-\d}}\int_{\R^2}|f_\tau|^2\lesssim  R^{O(\e)}\lambda(1)^2\sum_\g\|f_\g\|_2^2. 
\end{align*}
Using the same upper bound for $\frac{\lambda(1)^2}{\a^2}$ as in the previous case finishes the proof.

\end{proof}

\subsection{Bilinear reduction\label{bilred}}

We will present a broad/narrow analysis to show that Proposition \ref{mainprop} implies Theorem \ref{main}. In order to apply Proposition \ref{mainprop}, we must reduce to the case that $f$ has an $(R,\e)$-normalized distribution function $\lambda(\cdot)$. We demonstrate this through a series of pigeonholing steps.

\begin{proof}[Proposition \ref{mainprop} implies Theorem \ref{main}]

We will pigeonhole the $f_\g$ so that roughly, for any $s$-arc $\w$ of the parabola, the number
\[ \#\{\g:\g\cap \w\not=\emptyset,\quad f_\g\not=0\}\]
is either $0$ or relatively constant among $s$-arcs $\w$. For the initial step, write
\[ \{\tau_{N}:\exists\g\,\text{s.t.}\,f_\g\not=0,\,\,\g\subset\tau_{N}\}=\sum_{1\le \lambda\lesssim R^{\b}R^{-\e}}\Lambda_{N}(\lambda) \]
where $\lambda$ is a dyadic number, $\{\tau_{N}:\#\g\subset\tau_{N}\sim \lambda\}$, $\#\g\subset\tau_{N}$ means $\#\{\g\subset\tau_{N}:f_\g\not=0\}$, and $\#\g\subset\tau_{N}\sim\lambda$ means $\lambda\le \#\g\subset\tau_{N}<2\lambda$. Since there are $\lesssim \log R$ many $\lambda$ in the sum, there exists some $\lambda_{N}$ such that
\[ |\{x:|f(x)|>\a\}|\le C(\log R)|\{x:C(\log R)|\sum_{\tau_{N}\in\Lambda_{N}(\lambda_{N})}f_{\tau_{N}}(x)|>\a\}|. \]
Continuing in this manner, we have
\[ \{\tau_k:\exists \tau_{k+1}\in\Lambda_{k+1}(\lambda_{k+1})\,\text{s.t.}\,\tau_{k+1}\subset \tau_k\}=\sum_{1\le \lambda\le r_k}\Lambda_k(\lambda) \]
where $\Lambda_k(\lambda)=\{\tau_k:\exists \tau_{k+1}\in\Lambda_{k+1}(\lambda_{k+1})\,\text{s.t.}\,\tau_{k+1}\subset \tau_k\quad\text{and}\quad \#\g\subset\tau_k\sim\lambda\}$ and for some $\lambda_k$, 
\begin{align*} 
|\{x:(C(\log R))^{N-k}|&\sum_{\tau_{k+1}\in\Lambda_{k+1}(\lambda_{k+1})}f_{\tau_{k+1}}(x)|>\a\}|\\
&\qquad\qquad \le C(\log R)|\{x:(C(\log R))^{N-k+1}|\sum_{\tau_{k}\in\Lambda_{k}(\lambda_{k})}f_{\tau_{k}}(x)|>\a\}|  . 
\end{align*} 
Continue this process until we have found $\tau_1$, $\lambda_1$ so that
\begin{align*} 
|\{x:|f(x)|>\a\}|\le C^{\e^{-1}}(\log R)^{O(\e^{-1})}|\{x:C^{\e^{-1}}(\log R)^{O(\e^{-1})}|\sum_{\tau_{1}\in\Lambda_{1}(\lambda_{1})}f_{\tau_{1}}(x)|>\a\}|  . 
\end{align*} 
The function $\sum_{\tau_1\in\Lambda_{1}(\lambda_{1})}f_{\tau_{1}}$ now satisfies the hypotheses of Theorem \ref{main} and the property that $\#\g\subset\tau_k\sim\lambda_k$ or $\#\g\subset\tau_k=0$ for all $k$, $\tau_k$. It follows that the associated distribution function $\lambda(\cdot)$ of $\sum_{\tau_1\in\Lambda_{1}(\lambda_{1})}f_{\tau_{1}}$ is $(R,\e)$-normalized since 
\[ \lambda_m\sim \#\g\subset\tau_m=\sum_{\tau_k\subset\tau_m}\#\g\subset\tau_k\sim (\#\tau_k\subset\tau_m)(\lambda_k)  \]
where we only count the $\g$ or $\tau_k$ for which $f_\g$ or $f_{\tau_k}$ is nonzero. Now we may apply Proposition \ref{mainprop}. Note that since $\log R\le \e^{-1}R^\e$ for all $R\ge 1$, the accumulated constant from this pigeonholing process satisfies $C^{\e^{-1}}(\log R)^{O(\e^{-1})}\le C_\e R^\e$. It thus suffices to prove Theorem \ref{main} assuming that $f$ is $(R,\e)$-normalized.

Now we present a broad-narrow argument adapted to our set-up. Write $K=R^\d$ for some $\d>0$ which will be chosen later. Since $|f(x)|\le\sum_{\ell(\tau)=K^{-1}}|f_\tau(x)|$, there is a universal constant $C>0$ so that $|f(x)|>K^C\max_{\substack{\ell(\tau)=\ell(\tau')=K^{-1}\\\tau,\tau'\,\,\text{nonadj.}}}|f_\tau(x)f_{\tau'}(x)|^{1/2}$ implies $|f(x)|\le C\max_{\ell(\tau)=K^{-1}}|f_{\tau}(x)|$. If $|f(x)|\le K^C\max_{\substack{\ell(\tau)=\ell(\tau')=K^{-1}\\\tau,\tau'\,\,\text{nonadj.}}}|f_\tau(x)f_{\tau'}(x)|^{1/2}$ and $K^C\max_{\substack{\ell(\tau)=\ell(\tau')=K^{-1}\\\tau,\tau'\,\,\text{nonadj.}}}|f_\tau(x)f_{\tau'}(x)|^{1/2}\le C\max_{\ell(\tau)=K^{-1}}|f_{\tau}(x)|$, then $|f(x)|\le  C\max_{\ell(\tau)=K^{-1}}|f_{\tau}(x)|$. Using this reasoning, we obtain the first step in the broad-narrow inequality
\begin{align*}
    |f(x)|&\le C\max_{\ell(\tau)=K^{-1}}|f_{\tau}(x)|+K^C\max_{\substack{\ell(\tau)=\ell(\tau')=K^{-1}\\\tau,\tau'\,\text{nonadj.}\\ C\underset{\ell(\tau_0)=K^{-1}}{\max}|f_{\tau_0}(x)|\le K^{C}|f_{\tau}(x)f_{\tau'}(x)|^{1/2}}}|f_{\tau}(x)f_{\tau'}(x)|^{1/2} .
\end{align*}
Iterate the inequality $m$ times (for the first term) where $K^m\sim R^{1/2}$ to bound $|f(x)|$ by
\begin{align*}
    |f(x)|&\lesssim  C^m \max_{\ell(\tau)=R^{-1/2}}|f_{\tau}(x)|\\
    &\qquad+C^mK^C\sum_{\substack{R^{-1/2}<\Delta<1\\\Delta\in K^\N}} \max_{\substack{\ell(\tilde{\tau})\sim \Delta}}\max_{\substack{\ell(\tau)=\ell(\tau')\sim K^{-1}\Delta\\\tau,\tau'\subset\tilde{\tau},\,\,\text{nonadj.}\\
    C\underset{\substack{\ell(\tau_0)=K^{-1}\Delta\\ \tau_0\subset\tilde{\tau}}}{\max}|f_{\tau_0}(x)|\le K^C|f_{\tau}(x)f_{\tau'}(x)|^{1/2}}}|f_{\tau}(x)f_{\tau'}(x)|^{1/2} .
\end{align*}
Recall that our goal is to bound the size of the set
\[ U_\a=\{x\in\R^2:\a\le |f(x)|\}. \]
By the triangle inequality and using the notation $\theta$ for blocks $\tau$ with $\ell(\tau)=R^{-1/2}$
\begin{align}
\label{broadnar1}    &|U_\a|\le |\{x\in\R^2:\a\lesssim C^m\max_{\theta}|f_{\theta}(x)|\}|
    +\sum_{\substack{R^{-1/2}<\Delta<1\\\Delta\in K^\N}} \sum_{\substack{\ell(\tilde{\tau})\sim \Delta\\\ell(\tau)=\ell(\tau')\sim K^{-1}\Delta\\\tau,\tau'\subset\tilde{\tau},\,\,\text{nonadj.}}}|U_\a(\tau,\tau')|
\end{align}
where $U_\a(\tau,\tau')$ is the set
\[\{x\in\R^2:\a\lesssim (\log R)C^mK^C|f_{\tau}(x)f_{\tau'}(x)|^{1/2},\,\,C(|f_\tau(x)|+|f_{\tau'}(x)|)\le K^C|f_{\tau}(x)f_{\tau'}(x)|^{1/2} \}. \]
The first term in the upper bound from \eqref{broadnar1} is bounded trivially by $\frac{\lambda(R^{-1/2})^2}{\a^4}\sum_{\g}\|f_\g\|_2^2$. By the assumption that $\|f_\g\|_\infty\lesssim 1$ for every $\g$, we know that $|f_\tau|\lesssim R^{\b}$ for any $\tau$. Also assume without loss of generality that $\a>1$ (otherwise Theorem \ref{main} follows from $L^2$-orthogonality). This means that there are  $\sim\log R$ dyadic values of $\a'$ between $\a$ and $R^\b$ so by pigeonholing, there exists $\a'\in[\a/(C^mK^C),R^\b]$ so that
\[ |U_\a(\tau,\tau')|\lesssim (\log R+\log (C^mK^C))|\text{Br}_{\a'}(\tau,\tau')|
\]
where the set $\text{Br}_{\a'}(\tau,\tau')$ is defined in \eqref{broad2}. By parabolic rescaling, there exists an affine transformation $T$ so that $f_\tau\circ T=g_{\underline{\tau}}$ and $f_{\tau'}\circ T=g_{\underline{\tau}'}$ where $\underline{\tau}$ and $\underline{\tau'}$ are $\sim K^{-1}$-separated blocks in $\mc{N}_{\Delta^{-2}R^{-1}}(\P^1)$. Note that the functions $g_{\underline{\tau}}$ and $g_{\underline{\tau}'}$ inherit the property of being $(\Delta^2R,\e)$-normalized in the sense required to apply Proposition \ref{mainprop} in each of the following cases. 

\noindent{\underline{Case 1:}} Suppose that for some $\b'\in[\frac{1}{2},1]$, $\Delta^{-1}R^{-\b}=(\Delta^{2}R)^{-\b'}$. Then for each $\g\in\mc{P}(R,\b)$, $f_\g\circ T=g_{\underline{\g}}$ for some $\underline{\g}\in\mc{P}(\Delta^2R,\b')$. Applying Proposition \ref{mainprop} with functions $g_{\underline{\tau}}$ and $g_{\underline{\tau}'}$ and level set parameter $\a'$ leads to the inequality
\begin{align*} 
&|\text{Br}_{\a'}(\tau,\tau')|\le K^C\a' \}| \le C_{\e,\d} R^\e C^mK^{O(1)}\times\\
&\qquad \begin{cases} \frac{1}{(\a')^4}\underset{R^{-\b}<s<R^{-1/2}}{\max}\lambda(s^{-1}R^{-1})\lambda(s)\sum_{\g\subset\tilde{\tau}}\|f_\g\|_2^2\quad&\text{if}\quad
(\a')^2>\frac{\lambda(\Delta)^2}{\underset{s}{\max}\lambda(s^{-1}R^{-1})\lambda(s)}\\
\frac{\lambda(\Delta)^2}{(\a')^6}\sum_{\g\subset\tilde{\tau}}\|f_\g\|_2^2\quad&\text{if}\quad  (\a')^2\le \frac{\lambda(\Delta)^2}{\underset{s}{\max}\lambda(s^{-1}R^{-1})\lambda(s)} 
\end{cases}\end{align*}

\noindent{\underline{Case 2:}} Now suppose that $\Delta^{-1}R^{-\b}<(\Delta^2R)^{-1}$. Let $\tilde{\theta}$ be $\Delta^{-1}R^{-1}\times R^{-1}$ blocks and let $\underline{\tilde{\theta}}$ be $(\Delta^2R)^{-1}\times (\Delta^2R)^{-1}$ blocks so that $f_{\tilde{\theta}}\circ T=g_{\underline{\tilde{\theta}}}$. Let $B=\max_{\tilde{\theta}}|f_{\tilde{\theta}}|$ and divide everything by $B$ in order to satisfy the hypotheses $\|g_{\underline{\tilde{\theta}}}\|_\infty/B\le 1$ for all $\underline{\tilde{\theta}}$. Let $\tilde{\lambda}(s):=\lambda(\Delta s)/\lambda(\Delta^{-1}R^{-1})$ count the number of $\underline{\tilde{\theta}}$ intersecting an $s$-arc. 
In the case $(\a')^2>\frac{\tilde{\lambda}(1)B^2}{\max_s\tilde{\lambda}(s^{-1}(\Delta^2 R)^{-1})\tilde{\lambda}(s)}$ (with the maximum taken over $(\Delta^2R)^{-1}<s<(\Delta^2 R)^{-1/2}$), use Proposition \ref{mainprop} with functions $g_{\underline{\tau}}/B$ and $g_{\underline{\tau}'}/B$ and level set parameter $\a'/B$ to get the inequality

\[|\text{Br}_{\a'}(\tau,\tau')|\le C_{\e,\d} R^\e C^mK^{O(1)}
\frac{B^4}{(\a')^4}\underset{(\Delta^2 R)^{-1}<s<(\Delta^2 R)^{-1/2}}{\max}\tilde{\lambda}(s^{-1}(\Delta^2 R)^{-1})\tilde{\lambda}(s)\sum_{\tilde{\theta}\subset\tilde{\tau}}\|f_{\tilde{\theta}}\|_2^2/B^2. \]
 
Note that since $B\le \lambda(\Delta^{-1}R^{-1})$,
\[ B^2\underset{(\Delta^2 R)^{-1}<s<(\Delta^2 R)^{-1/2}}{\max}\tilde{\lambda}(s^{-1}(\Delta^2 R)^{-1})\tilde{\lambda}(s)\le\underset{\Delta^{-1} R^{-1}<s<R^{-1/2}}{\max}{\lambda}(s^{-1} R^{-1}){\lambda}(s) \]
and 
\[ \frac{\tilde{\lambda}(1)^2B^2}{\underset{s}{\max}\tilde{\lambda}(s^{-1}(\Delta^2 R)^{-1})\tilde{\lambda}(s)} 
\le \frac{\lambda(\Delta)^2\lambda(\Delta^{-1}R^{-1})^2}{\underset{\Delta^{-1} R^{-1}<s<R^{-1/2}}{\max}{\lambda}(s^{-1} R^{-1}){\lambda}(s)}\le \lambda(\Delta^{-1}R^{-1})\lambda(\Delta)
. \] 
Then in the case $(\a')^2\le\frac{\tilde{\lambda}(1)B^2}{\max_s\tilde{\lambda}(s^{-1}(\Delta^2 R)^{-1})\tilde{\lambda}(s)}$, compute directly that 
\begin{align*}
&(\a')^4|\{x\in\R^2:\a'\sim|f_{\tau}(x)f_{\tau'}(x)|^{1/2},\,\,(|f_\tau(x)|+|f_{\tau'}(x)|)\le K^C\a' \}| \\
&\quad \lesssim \lambda(\Delta^{-1}R^{-1})\lambda(\Delta)\int_{\R^2}(|f_\tau|^2+|f_{\tau'}|^2)\lesssim\underset{\Delta^{-1}R^{-1}<s<R^{-1/2}}{\max}{\lambda}(s^{-1}R^{-1}){\lambda}(s)\sum_{\g\subset\tilde{\tau}}\|f_{\g}\|_2^2.
\end{align*}
Using also that $\sum_{\tilde{\theta}\subset\tilde{\tau}}\|f_{\tilde{\theta}}\|_2^2\le\sum_{\g\subset\tilde{\tau}}\|f_\g\|_2^2$, the bound for Case 2 is
\begin{align*} 
&|\{x\in\R^2:\a'\sim|f_{\tau}(x)f_{\tau'}(x)|^{1/2},\,\,(|f_\tau(x)|+|f_{\tau'}(x)|)\le K^C\a' \}| \\
&\quad \le C_{\e,\d} R^\e C^mK^{O(1)}
\frac{1}{(\a')^4}\underset{R^{-\b}<s<R^{-1/2}}{max}{\lambda}(s^{-1}(\Delta^2 R)^{-1}){\lambda}(s)\sum_{\g\subset\tilde{\tau}}\|f_{\g}\|_2^2.
\end{align*}

It follows from \eqref{broadnar1} and the combined Case 1 and Case 2 arguments above that
\begin{align*} 
&|U_\a| \le C_{\e,\d} R^\e C^mK^{O(1)}\times\\
&\qquad \begin{cases} \frac{1}{\a^4}\underset{R^{-\b}<s<R^{-1/2}}{\max}\lambda(s^{-1}R^{-1})\lambda(s)\sum_{\g}\|f_\g\|_2^2\quad&\text{if}\quad
\a>\frac{\lambda(1)^2}{\underset{s}{\max}\lambda(s^{-1}R^{-1})\lambda(s)}\\
\frac{\lambda(1)^2}{\a^6}\sum_{\g}\|f_\g\|_2^2\quad&\text{if}\quad  \a^2\le \frac{\lambda(1)^2}{\underset{s}{\max}\lambda(s^{-1}R^{-1})\lambda(s)} 
\end{cases}.\end{align*}
Recall that $K^m\sim R^{-1/2}$ and $K=R^{\d}$ so that $C_{\e,\d} R^\e C^mK^{O(1)}\le C_{\e,\d} R^\e C^{O(\d^{-1})}R^{O(1)\d}$. Choosing $\d$ small enough so that $R^{O(1)\d}\le R^\e$ finishes the proof. 

\end{proof}

\bibliographystyle{alpha}
\bibliography{AnalysisBibliography}

\end{document}